\documentclass[11pt,a4paper]{article}
\usepackage{amsmath}
\usepackage[toc,page]{appendix}
\usepackage{amssymb}
\usepackage{bbm}
\usepackage{amsthm}
\usepackage{enumerate}
\usepackage{graphicx,epstopdf}
\usepackage{bold-extra}
\usepackage[margin=1.97cm]{geometry}
\usepackage{hyperref}
\usepackage{color}
\newtheorem{theorem}{Theorem}
\newtheorem{remark}{Remark}
\newtheorem{definition}{Definition}

\newtheorem{lemma}{Lemma}
\newtheorem{proposition}{Proposition}

\newtheorem{Op}{Open problem}
\numberwithin{equation}{section}
\renewcommand{\div}{\operatorname{div\,}}
\newcommand{\curl}{\operatorname{curl}}
\newcommand{\eps}{\varepsilon}
\newcommand{\R}{\mathbb R}

\title{External boundary control of the motion of a rigid body immersed in a perfect two-dimensional fluid}
\author{Olivier Glass\footnote{CEREMADE, UMR CNRS 7534, Universit\'e Paris-Dauphine, PSL Research University,
Place du Mar\'echal de Lattre de Tassigny, 75775 Paris Cedex 16, France}, J\'ozsef J. Kolumb\'an 
\footnote{Institut f\"ur Mathematik, Universit\"at Leipzig, D-04109, Leipzig, Germany.},
Franck Sueur\footnote{Univ. Bordeaux, CNRS, Bordeaux INP, IMB, UMR 5251,  F-33400, Talence, France  $\&$ Institut  Universitaire de France. }}

\date{}

\begin{document}

\maketitle

\begin{abstract}
We consider the motion of a rigid body immersed in a two-dimensional irrotational perfect incompressible fluid. 
The fluid is governed by the Euler equation, while the trajectory of the solid is given by Newton's equation, the force term corresponding to the fluid pressure on the body's boundary only.
The system is assumed to be confined in a bounded domain with an impermeable condition on a part of the external boundary.
The issue considered here is the following: is there an appropriate boundary condition on the remaining part of the external boundary (allowing some fluid going in and out the domain) 
such that the immersed rigid body is driven from  some given initial  position and velocity to some final position (in the same connected component of the set of possible positions as the initial position) and velocity in a given positive time, without touching the external boundary ?
In this paper we provide a positive answer to this question thanks to an  impulsive control strategy. To that purpose 
we make use of a reformulation of the solid equation into an ODE of geodesic form, with some force terms due to the circulation around the body (as in \cite{GMS}) and some extra terms  here due to the external boundary control.
\end{abstract}

\newpage
\tableofcontents
\newpage
%
%
%
%
%
%
%
%
%
\section{Introduction and main result}
\label{INTRO}

\subsection{The model without control}
\label{INTRO-WC}

A simple model of fluid-solid evolution is that of a single rigid body surrounded by a perfect incompressible fluid. Let us describe this system. We consider a two-dimensional bounded, open, smooth and simply connected\footnote{The condition of simple connectedness is actually not essential and one could generalize the present result to the case where $\Omega$ is merely open and connected at the price of long but straightforward modifications.}  domain 
$\Omega \subset \mathbb{R}^{2}$. 
The domain $\Omega$  is composed of two disjoint parts: the open part ${\mathcal F}(t)$ filled with fluid and the closed part ${\mathcal S}(t)$ representing the solid. These parts depend on time $t$. Furthermore, we assume that ${\mathcal S}(t)$ is also smooth and simply connected. 
On the fluid part ${\mathcal F}(t)$, the velocity field ${u}:\left\{(t,x):\ t\in[0,T],\ x\in\overline{\mathcal{F}(t)} \right\}\rightarrow \mathbb{R}^2$ and the pressure field ${\pi}:\left\{(t,x):\ t\in[0,T],\ x\in\overline{\mathcal{F}(t)} \right\} \rightarrow \mathbb{R}$ satisfy the incompressible Euler equation:
\begin{eqnarray}\label{eu}
\begin{split}
\frac{\partial u}{\partial t}+(u\cdot\nabla)u + \nabla \pi =0 \   \text{ and } \ \div u = 0
\ \text{ for } \ t\in[0,T]   \text{ and  } x \in \mathcal{F}(t) .
\end{split}
\end{eqnarray}
We consider impermeability boundary conditions, namely, on the solid boundary, the normal velocity coincides with the solid normal velocity
\begin{equation} \label{bc1}
u \cdot n = u_{S}\cdot n\ \text{on}\ \partial \mathcal{S}(t),
\end{equation}
where $u_{S}$ denotes the solid velocity described below,
while on the outer part of the boundary we have
\begin{equation} \label{bc2}
u \cdot n =0\ \text{on}\ \partial \Omega,
\end{equation}
where ${n}$ is the unit outward normal vector on $\partial\mathcal{F}(t)$.
The solid ${\mathcal S}(t)$ is obtained by a rigid movement from ${\mathcal S}(0)$, and one can describe its position by the center of mass, ${h}(t)$, and the angle variable with respect to the initial position, ${\vartheta}(t)$. 
Consequently, we have
\begin{align}
\mathcal{S}(t)=h(t)+R(\vartheta(t)) (\mathcal{S}_0-h_0),
\end{align} 
where $h_0$ is the center of mass at initial time, and
$${R}(\vartheta)= \left( \begin{array}{cc}
\cos \vartheta & -\sin \vartheta\\
\sin \vartheta & \cos \vartheta\end{array} \right) .$$
Moreover the solid velocity is hence given by
\begin{align}
u_S (t,x)={h}'(t)+\vartheta'(t) (x-h(t))^{\perp},
\end{align} 
where for $x=(x_1,x_2)$ we denote $x^\perp  = (-x_2,x_1).$

The solid evolves according to Newton's law, and is influenced by the fluid's pressure on the boundary:
\begin{equation} \label{newt}
m  h'' (t) =  \int_{ \partial \mathcal{S} (t)} \pi \, n \, \, d\sigma  \ \text{ and } \
\mathcal{J}  \vartheta'' (t) =    \int_{ \partial   \mathcal{S} (t)} \pi \, (x-  h (t) )^\perp  \cdot n \, d \sigma.
\end{equation}
Here the constants $m>0$ and $\mathcal{J}>0$ denote respectively the mass and the
moment of inertia of the body, where the 
fluid is supposed to be homogeneous of density 1, without loss of generality.
Furthermore, the circulation around the body is constant in time, that is
\begin{align}\label{kelv}
\int_{\partial\mathcal{S}(t)} u(t) \cdot \tau \, d\sigma = \int_{\partial\mathcal{S}_0} u_0 \cdot \tau \, d\sigma = \gamma \in \mathbb{R},\ \forall t\geq 0,
\end{align}
due to Kelvin's theorem, where ${\tau}$ denotes the unit counterclockwise tangent vector.

\begin{figure}[ht]
\centering
\resizebox{1\linewidth}{!}{\includegraphics[clip, trim=-1.25cm 17cm 0.5cm 4cm, width=1.00\textwidth]{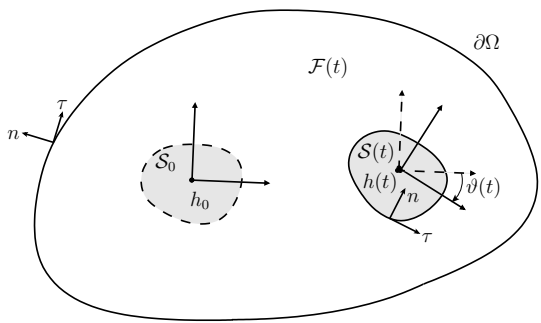}}
\caption{The domains $\Omega$, $\mathcal{S}(t)$ and $\mathcal{F}(t)=\Omega\setminus\mathcal{S}(t)$}
\end{figure}

The Cauchy problem for this system with initial data
\begin{eqnarray}\label{ic}
\begin{split}
u|_{t=0}=u_{0}\text{ for } x\in\mathcal{F}(0),\\
h(0)=h_{0},\ h'(0)=h'_0,\ \vartheta(0)=0,\ \vartheta'(0)=\vartheta'_0,
\end{split}
\end{eqnarray}
is now well-understood, see e.g. \cite{GLS-Mass,GS-Uniq,ht,ort1,ort2}. Furthermore, the 3D case has also been studied in \cite{GST,ros}. Note in passing that it is our convention used throughout the paper that $\vartheta(0)=0$.

In this paper, we will furthermore assume that the fluid is irrotational at the initial time, that is $\text{curl } u_0 = 0$ in $\mathcal{F}(0)$, which implies that it stays irrotational at all times, due to Helmholtz's third theorem, i.e.
\begin{align}\label{helm}
\text{curl } u  = 0\text{ for } x\in\mathcal{F}(t),\ \forall t\geq 0.
\end{align}
\subsection{The control problem and the main result}
\label{INTRO-C}

We are now in position to state our main result.

Our goal is to investigate the possibility of controlling the solid by means of a boundary control acting on the fluid. Consider  $\Sigma$ a nonempty, open part of the outer boundary $\partial \Omega$.
Suppose that one can choose some non-homogeneous boundary conditions on $\Sigma$. One natural possibility is due to Yudovich (see \cite{Yud}), which consists in prescribing on the one hand the normal velocity on $\Sigma$, i.e. choosing some function ${g}\in C_{0}^{\infty}([0,T]\times\Sigma)$ with $\int_\Sigma g = 0$ and imposing that
\begin{equation} \label{yud1}
u(t,x)\cdot n(x)=g(t,x)\ \text{on}\ [0,T]\times\Sigma,
\end{equation}
while on the rest of the boundary we have the usual impermeability condition
\begin{equation} \label{yud3}
u \cdot n =0\ \text{on}\ [0,T] \times (\partial \Omega \setminus \Sigma),
\end{equation}
and on the other hand the vorticity on the set $\Sigma^-$ of points of $[0,T] \times \Sigma$ where the velocity field points inside $\Omega$. Note that $\Sigma^-$ is deduced immediately from $g$.

Since we are interested in the vorticity-free case, we will actually consider here a null control in vorticity, that is
\begin{equation} \label{yud2}
\curl u(t,x)=0\ \text{on}\ \Sigma^{-} =\{(t,x)\in [0,T]\times\Sigma \  \text{ such that }  \ u(t,x)\cdot n(x)<0\}.
\end{equation}
Condition \eqref{yud2} enforces the validity of  \eqref{helm} as in the uncontrolled setting despite the fact that some fluid is entering the domain.

The general question of this paper is how to control the solid's movement by using the above boundary control (that is, the function $g$). In particular we raise the question of driving the solid from a given position and a given velocity to some other prescribed position and velocity. Remark that we cannot expect to control the fluid velocity in the situation described above: for instance, Kelvin's theorem gives an invariant of the dynamics, regardless of the control.

Throughout this paper we will only consider solid trajectories which stay away from the boundary. Therefore we introduce
$$\mathcal{Q}=\{ q := (h,\vartheta)\in\Omega\times\mathbb{R}:\ d(h+R(\vartheta)(\mathcal{S}_0-h_0),\partial\Omega)>0\}.$$
Furthermore, let us from here on set
$$\mathcal{D}_T:=\left\{(t,x):\ t\in[0,T],\ x\in\overline{\mathcal{F}(t)} \right\},$$
where we have omitted from the notation the dependence on $\mathcal{F}(\cdot)$, and therefore on the unknown $(h,\vartheta)(\cdot)$.

The main result of this paper is the following statement.
\begin{theorem} \label{main}
Let $T>0$.
Consider $\mathcal{S}_0\subset\Omega$ bounded, closed, simply connected with smooth boundary, which is not a disk, and $u_{0}\in C^{\infty}(\overline{\mathcal{F}(0)};\mathbb{R}^{2})$, $\gamma\in\mathbb{R}$, 
$q_0=(h_{0},0), q_1 =(h_{1},\vartheta_1) \in \mathcal{Q}$, 
$h'_0,h'_1\in\mathbb{R}^2,\vartheta'_0,\vartheta'_1\in\mathbb{R}$, such that $(h_0,0)$ and $(h_1,\vartheta_1)$ belong to the same connected component of $\mathcal{Q}$ and
\begin{gather*}
\div u_0=\curl u_0=0 \text{ in }\mathcal{F}(0), \ 
u_0 \cdot n =0 \text{ on } \partial \Omega, \\
u_0 \cdot n = (h'_0+\vartheta'_0 (x-h_0)^{\perp}) \cdot n \text{ on } \partial \mathcal{S}_0, \ 
\int_{\partial\mathcal{S}_0} u_0 \cdot \tau \, d\sigma =\gamma.
\end{gather*}
Then there exists a control $g\in C_{0}^{\infty}((0,T)\times\Sigma)$ and a solution
 $(h,\vartheta,u)\in C^{\infty}([0,T];\mathcal{Q})\times C^{\infty}(\mathcal{D}_T;\mathbb{R}^{2})$ 
 to  
 \eqref{eu}, \eqref{bc1}, \eqref{newt},  \eqref{kelv}, \eqref{ic}, \eqref{helm}, \eqref{yud1}, \eqref{yud3}, which satisfies
$(h,h',\vartheta,\vartheta')(T)=(h_1,h'_1,\vartheta_1,\vartheta'_1)$.
\end{theorem}
%
%
%
\begin{figure}[ht]
\centering
\resizebox{1\linewidth}{!}{\includegraphics[clip, trim=-1.25cm 16cm 0.5cm 4cm, width=1.00\textwidth]{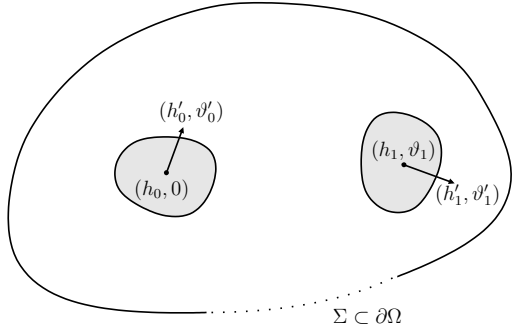}}
\caption{The initial and final positions and velocities in the control problem}
\end{figure}
\begin{remark}
  \label{smallflux}
In Theorem~\ref{main} the control $g$ can be chosen with an arbitrary small total flux through $\Sigma^{-}$, that is for any $T>0$, 
 for any $\nu > 0$,  there exists a control $g$ and a solution $(h,\vartheta,u)$ satisfying the properties of Theorem  \ref{main}  and such that moreover 
$$ \left\vert \int_0^T \int_{\Sigma^{-}} \, u \cdot n  \, \, d\sigma  dt \right\vert < \nu .$$
See Section~\ref{REM-smallflux} for more explanations.
Let us mention that such a small flux condition cannot be guaranteed in the results \cite{Coron:EC,OG-Addendum,OG:E3} regarding the controllability of the Euler equations.
\end{remark}

When $\mathcal{S}_0$ is a disk, the second equation in (\ref{newt}) becomes degenerate, so it needs to be treated separately. For instance, in the case of a homogeneous disk, i.e. when the center of mass coincides with the center of the disk and we have $(x-h(t))^\perp \cdot n =0$, for any $x\in \partial\mathcal{S}(t),\ t\geq 0$, hence we cannot control $\vartheta$. However,
we have a similar result for controlling the center of mass $h$.
\begin{theorem}\label{dmain}
Let $T>0$.
Given a homogeneous disk $\mathcal{S}_0\subset\Omega$, $u_{0}\in C^{\infty}(\overline{\mathcal{F}(0)};\mathbb{R}^{2})$, $\gamma\in\mathbb{R}$, 
$h_{0},h_{1}\in\Omega,\ h'_0,h'_1\in\mathbb{R}^2,$ such that $(h_0,0)$ and $(h_1,0)$ are in the same connected component of $\mathcal{Q}$, and
$\div u_0=\curl u_0=0\ \text{in}\ \mathcal{F}(0)$, 
$u_0 \cdot n =0$ on $ \partial \Omega$, 
$u_0 \cdot n = h'_0\cdot n$ on $ \partial \mathcal{S}_0$, 
$\int_{\partial\mathcal{S}_0} u_0 \cdot \tau \, d\sigma =\gamma$,
there exists $g\in C_{0}^{\infty}((0,T) \times\Sigma)$ and a solution $(h,u)$ in $C^{\infty}([0,T];\Omega)\times C^{\infty}(\mathcal{D}_T;\mathbb{R}^{2})$ of (\ref{eu}), (\ref{bc1}), (\ref{newt}), (\ref{kelv}), (\ref{helm}), (\ref{yud1}), (\ref{yud3}), (\ref{yud2}) with initial data $(h_0,h'_0,u_0)$, which satisfies
$(h,h')(T)=(h_1,h'_1)$.
\end{theorem}
The proof is similar to that of Theorem \ref{main}, with the added consideration that $(x-h(t))^\perp \cdot n =0$, for any $x\in \partial\mathcal{S}(t),\ t\geq 0$. We therefore omit the proof.
In the case where the  disk is non-homogeneous the analysis is technically more intricate already in the uncontrolled setting, see \cite{GMS}, and we will let aside this case in this paper.

\ \par \noindent
{\bf References.} Let us mention a few results of boundary controllability of a fluid alone, that is without any moving body. The problem is then finding a boundary control which steers the fluid velocity from $u_0$ to some prescribed state $u_1$.
For the incompressible Euler equations 
small-time global exact boundary controllability  has been obtained in \cite{Coron:EC,OG:E3} in the  2D, respectively 3D case. 
This result has been recently extended to the case of the incompressible Navier-Stokes equation with Navier slip-with-friction boundary conditions in  \cite{CMS}, see also \cite{CMS2}  for a gentle exposition.
Note that the proof there relies on the previous results for the Euler equations by means of a rapid and strong control which drives the system in a high Reynolds regime. 
This strategy was initiated in  \cite{Coron:NS},  where an interior controllability result was already established. 
For ``viscous fluid + rigid body'' control systems (with Dirichlet boundary conditions), local controllability results have already been obtained in both 2D and 3D, see e.g. \cite{BG, BO, IT}. These results rely on Carleman estimates on the linearized equation, and consequently on the parabolic character of the fluid equation. 

A different type of fluid-solid control result can be found in \cite{G-R}, where the fluid is governed by the two-dimensional Euler equation. However in this paper the control is located on the solid's boundary which makes the situation quite different. 

Actually, the results of Theorem \ref{main}  and  Theorem \ref{dmain} can rather be seen as some extensions to the case of an immersed body of 
the results \cite{OGTH,OGTH2,OGTH3}  concerning Lagrangian controllability of the incompressible Euler and Stokes equations, where the control takes the same form as here. 

\subsection{Generalizations and open problems}
First, as we mentioned before, using the techniques of this paper, the result could be straightforwardly generalized for non simply connected domains. One could also manage in the same way the control of several solids (the reader may in particular see that the argument using Runge's theorem in Section~\ref{S4} is local around the solid). 

We would also like to underline that the absence of vorticity is not central here. This may surprise the reader acquainted with the Euler equation, but actually following the arguments of Coron \cite{Coron:EC,Coron:NS}, one knows how to control the full model when one can control the irrotational one. This is by the way the technique that we use to take care of the circulation $\gamma$ (see in particular Section~\ref{RELO}). But the presence of vorticity makes a lot of complications from the point of view of the initial boundary problem, in particular for what concerns the uniqueness issue, see Yudovich \cite{Yud}. To avoid these unnecessary technical complications, we restrain ourselves to the irrotational problem. But the full problem could undoubtedly be treated in the same way.

Furthermore, one might ask the question whether or not it is possible to control with a reduced number of controls, i.e. to only look for controls $g$ which take the form of a linear combination of some a priori given controls $\{g_i\}_{i=1,\ldots,I}$, which may depend on the geometry, but not the initial or final data of the control problem. We consider that our methods can be adapted to prove such a result, in particular since in Section \ref{RELO} we prove that Theorem \ref{main} follows from a simpler result, Theorem \ref{lmain}, where the solid displacement, the solid velocities and the circulation are small. It then suffices to discretize the control with respect to the parameters $(h_0,h'_0,\vartheta_0,\vartheta'_0)$, $(h_1,h'_1,\vartheta_1,\vartheta'_1)$ and $\gamma$. This does not pose a problem since our control is actually constructed continuously with respect to these parameters, so one may apply a compactness argument. 
However, the set of controls $\{g_i\}_{i=1,\ldots,I}$ will depend on the parameter $\delta>0$ from Theorem \ref{lmain}, used to restrict the set of admissible positions $\mathcal{Q}$ to the set $\mathcal{Q}_\delta$ defined in \eqref{qd}. This subtlety is due to the fact that the closure of $\mathcal{Q}$ also contains points where the solid touches the outer boundary, while this is no longer the case with $\mathcal{Q}_\delta$ for a given fixed $\delta>0$, and we use this for the compactness argument mentioned above.

\par
There remain also many open problems.

Considering the recent progresses on the controllability in the viscous case, a natural question is whether or not the results in this paper could be adapted to the case where a rigid body is moving in a fluid driven by the incompressible Navier-Stokes equation. In \cite{NSs} we extend the analysis performed here to prove the small-time global controllability of the motion of a rigid body in a viscous incompressible fluid, driven by the incompressible Navier-Stokes equation, in the case where Navier slip-with-friction boundary conditions are prescribed at the interface between the fluid and the solid. However, the case of Dirichlet boundary conditions remains completely open.

Let us mention the following open problem regarding the motion planning of a rigid body immersed in an inviscid incompressible irrotational flow. 
\begin{Op}\label{motionplanning}
Let $T>0$, $(h_{0},0)$ in $\mathcal{Q}$, 
 $\xi$ in $C^{2}([0,T]; \mathcal{Q})$, with  $\xi(0)= (h_0,0)$.
 Let us decompose $\xi'(0)$ into $\xi'(0)= (h'_0 , \vartheta'_0)$. 
 Consider $\mathcal{S}_0\subset\Omega$ bounded, closed, simply connected with smooth boundary, which is not a disk, 
  $\gamma\in\mathbb{R}$, 
 and 
  $u_{0}\in C^{\infty}(\overline{\mathcal{F}(0)};\mathbb{R}^{2})$ such that 
  $\div u_0=\curl u_0=0 $ in $\mathcal{F}(0)$,
$u_0 \cdot n =0$ on $ \partial \Omega$, 
$u_0 \cdot n = (h'_0+\vartheta'_0 (x-h_0)^{\perp})\cdot n$ on $\partial \mathcal{S}_0$ and 
$\int_{\partial\mathcal{S}_0} u_0 \cdot \tau \, d\sigma =\gamma$.
 Do there exist $g\in C_{0}([0,T]\times\Sigma)$ and a solution
 $(h,\vartheta,u)\in C^{2}([0,T];\mathcal{Q})\times C^{1}(\mathcal{D}_T;\mathbb{R}^{2})$
 to  
 (\ref{eu}), (\ref{bc1}), (\ref{newt}), (\ref{kelv}), (\ref{ic}), (\ref{helm}), (\ref{yud1}),  (\ref{yud3}), which satisfies
$  \xi = (h,\vartheta) $?
\end{Op}
Even the approximate motion planning in $C^2$, i.e. the same statement as above but with 
$  \|  \xi - (h,\vartheta) \|_{C^{2}([0,T])} \leq \eps$ (with $\varepsilon>0$ arbitrary) instead of $\xi = (h,\vartheta)$, is an open problem. 

Furthermore, in this paper we have ignored any possible thermodynamic effect in the model, however, it would be a natural question to ask how our results could be generalized to the case when the fluid is heat-conductive.


\subsection{Plan of the paper and main ideas behind the proof of Theorem \ref{main}} 
 The paper is organized as follows.
 
 \bigbreak
In Section \ref{NODE} we first recall from  \cite{GMS}  a reformulation of the Newton equations \eqref{newt} as an ODE in the uncontrolled case and then extend it to the case with control. 

To be more precise, denoting $q:=(h,\vartheta)$ and considering a manifold of admissible positions $\mathcal{Q}$ (to be defined later), the authors proved in \cite{GMS} that there exist a field $\mathcal{M}:\mathcal{Q}\to S^{++}(\mathbb{R}^3)$ of symmetric positive-definite matrices and smooth fields $E,B:\mathcal{Q}\to\mathbb{R}^3$, such that the fluid-solid system is equivalent to the following ODE in $q$:
$$\mathcal{M}(q)q''+\langle \Gamma(q),q',q'\rangle=\gamma^2 E(q)+\gamma q' \times B(q),$$
where $\Gamma(q)$ is a bilinear symmetric mapping, given by the so-called Christoffel symbols of the first kind:
$$\Gamma^k_{i,j}  =\frac{1}{2}\left( \frac{\partial (\mathcal{M})_{k,j} }{\partial q_i}  + \frac{\partial (\mathcal{M})_{k,i} }{\partial q_j} -\frac{\partial (\mathcal{M})_{i,j} }{\partial q_k} \right)  .$$
In particular, the case with zero circulation represents the fact that the particle $q$ is moving along the geodesics associated with the Riemannian metric induced on $\mathcal{Q}$ by the so-called total inertia matrix $\mathcal{M}$.

We extend the above result to the case with control $g\in C_{0}^{\infty}([0,T]\times\Sigma)$, to find that $q$ satisfies the following ODE:
\begin{align}\label{renform00}
\mathcal{M}(q)q''+\langle \Gamma(q),q',q'\rangle=\gamma^2 E(q)+\gamma q' \times B(q)+ F_1 (q,q', \gamma) [ \alpha] +   F_2  (q) [\partial_t\alpha],
\end{align}
where $F_1$ and $F_2$ are regular, respectively $\alpha$ is defined as the unique smooth solution of the Neumann problem
\begin{align}\label{pot00}
\Delta \alpha =0\ \text{in}\  \mathcal{F}(t) \quad    \text{ and } \quad 
\partial_{n} \, \alpha=g\mathbbm{1}_{\Sigma}\ \text{on}\ \partial\mathcal{F}(t) ,
\end{align}
with zero mean.

Note that in both cases above, the fluid velocity $u$ can be recovered by solving some simple elliptic PDEs.

\bigbreak
   
In Section \ref{RELO} we prove that Theorem \ref{main} can be deduced from a simpler result, namely Theorem \ref{lmain}, where  the solid displacement, the initial and final solid velocities and the circulation are assumed to be small. 

This is achieved on one hand by using the usual time-rescale properties of the Euler equation in order to pass from arbitrary solid velocities and circulation to small ones. More precisely, if $u(t,\cdot)$ is a solution to the Euler equation on $[0,T]$, then for any $\lambda>0$, $u^\lambda(t,\cdot):=\frac{1}{\lambda}u\left( \frac{t}{\lambda},\cdot\right)$ is a solution to the Euler equation on the time interval $[0,\lambda T]$. The corresponding scaling for the initial and final solid velocities and the circulation associated with $u^\lambda$ becomes  $\frac{q'_0}{\lambda}$,  $\frac{q'_1}{\lambda}$ and $ \frac{\gamma}{\lambda}$. Hence, if one can find a solution with small initial and final velocities and small circulation on $[0,T]$, one can pass to the arbitrary (or large) case on $[0,\lambda T]$ with $\lambda \in (0,1)$ small enough, thus obtaining the controllability result in smaller time. There are multiple possibilities for using up the remaining time from $\lambda T$ to $T$, and we give one in Section \ref{RELO}, relying on the time-reversal properties of the Euler equation.

On the other hand,  one may use a compact covering argument to pass from the case when $q_0$ and $q_1$ are remote to the case when their distance is small.

\bigbreak

In Section \ref{REAP}  we prove that  another reduction is possible, as we prove that an approximate controllability result (rather than an exact one), namely Theorem \ref{approx}, allows to deduce Theorem \ref{lmain}.

Indeed, if instead of $(q,q')(T)=(q_1,q'_1)$ one only has $\|(q,q')(T)-(q_1,q'_1)\|\leq \eta$, for $\eta>0$ small enough, then it is possible to pass to exact controllability by using a Brouwer-type topological argument. However, for such a result to be applied, one has to make sure that the map $(q_1,q'_1)\mapsto (q,q')(T)$ is well-defined and continuous for 
$(q_1,q'_1)$ in some small enough ball, which we will indeed achieve during our construction.

\bigbreak
 
Section \ref{proof-approx} is devoted to the proof of Theorem \ref{approx} and is the core of the paper. In order to achieve the aforementioned approximate controllability, we rely on the following strategy.

Suppose we have $\gamma=0$ (if this is not the case, one can at least expect to be close in some sense to the case without circulation when $\gamma$ is small enough), and suppose that we can find some appropriate control $g\in  C^{\infty}_0 ([0,T]; \mathcal C)$
such that the term $F_1 (q,q', 0) [ \alpha] +   F_2  (q) [\partial_t\alpha]$ in \eqref{renform00} behaves approximately like $v_0 \delta_0(t) + v_1 \delta_T(t)$, for any given $v_0,v_1\in\mathbb{R}^3$, where $\delta_0$ and $\delta_T$ denote the Dirac distributions at time $t=0^+$, respectively $t=T^-$.

Then, \eqref{renform00} is going to be close (in an appropriate sense) to the following formal toy model:
 \begin{align}\label{toy00}
\mathcal{M}(\tilde{q})  \tilde{q}''+\langle \Gamma(\tilde{q}),\tilde{q}',\tilde{q}'\rangle = v_0 \delta_0 + v_1 \delta_T,
\end{align}
and controlling \eqref{renform00} (at least approximately) reduces to controlling \eqref{toy00} by using the vectors $v_0,v_1\in\mathbb{R}^3$ as our control. In fact, we consider a control of the form 
\begin{align}\label{ctrl00}
g(t,x)=\beta_0(t)\bar{g}_0(x)+\beta_1(t)\bar{g}_1(x),
\end{align}
where the functions $\beta_0,\beta_1$ are chosen as square roots of sufficiently close smooth approximations of $\delta_0,\delta_T$ (since it turns out that $F_1$ depends quadratically on $\alpha$, and by consequence also on $g$, see \eqref{pot00}), and with some appropriate functions $\bar{g}_0,\bar{g}_1$.

Let us quickly explain how the controllability of the toy model \eqref{toy00} can be established. Given $q_0,q_1\in\mathcal{Q}$, there exists (at least in the case when $q_0$ and $q_1$ are sufficiently close, hence the arguments of Section \ref{RELO}) a geodesic associated with the Riemannian metric induced on $\mathcal{Q}$ by $ \mathcal{M}$, which connects $q_0$ with $q_1$. More precisely, there exists a unique smooth function $\bar{q}$ satisfying 
\begin{align}\label{geod00}
\begin{split}
 \mathcal{M}(\bar{q})\bar{q}''+\langle\Gamma (\bar{q}),\bar{q}',\bar{q}'\rangle=0  \text{  on } [0,T], 
  \text{ with  } \bar{q}(0)= q_0,\ \bar{q}(T)={q}_1.
\end{split}
\end{align}
So, one can arrive at the desired final position $q_1$, but a priori the final velocity $\bar{q}'(T)$ differs from $q'_1$, furthermore even the initial velocity $\bar{q}'(0)$ differs from $q'_0$.

Then, controlling the solution $\tilde{q}$ of \eqref{toy00} from $(q_0,q'_0)$ to $(q_1,q'_1)$ just amounts to setting $v_0 :=\mathcal{M}(q_0) (\bar{q}'(0) - q'_0 )$ and $v_1 := -  \mathcal{M}(q_1) (\bar{q}'(T) - q'_1 )$, which transforms the initial and final velocities $\tilde{q}'(0)$ and $\tilde{q}'(T)$ exactly to the desired velocities in order to achieve controllability.

\bigbreak

In Section \ref{S3} we prove a Proposition that is important for Theorem \ref{approx}, namely that the whole system will behave like the toy model above, in a certain regime (and in particular for small $\gamma$). This relies on some appropriate estimations of the terms $F_1$, $F_2$ and some Gronwall-type arguments.

\bigbreak

Section \ref{S4} explains how one can construct the control by means of complex analysis: it can be considered as the cornerstone of our control strategy. It is here that we construct the spacial parts $\bar{g}_0,\bar{g}_1$ of our control $g$ from \eqref{ctrl00}, in function of $v_0, v_1$.

%
%
%
%
%
%
\section{Reformulation of the solid's equation into an ODE}
\label{NODE}

In this section we establish a reformulation of the Newton equations (\ref{newt}) as an ODE for the three degrees of freedom of the rigid body with coefficients obtained by solving some elliptic-type problems on a domain depending on the solid position.
 Indeed the fluid velocity can  be recovered from the  solid position and velocity by an elliptic-type problem, so that the fluid state may be seen as  solving an auxiliary steady problem, where time only appears as a parameter, instead of the evolution equation  (\ref{eu}). 
  The Newton equations can therefore be rephrased as a second-order differential equation on the solid position whose coefficients are determined by the auxiliary fluid problem. 
  
  Such a reformulation in the case without boundary control was already achieved in \cite{GMS} and we will start by recalling this case in Section \ref{NODEA}, cf. Proposition \ref{reformC} below. 
 A crucial fact in the analysis is that in the ODE reformulation 
the pre-factor of the body's accelerations is the sum of  the inertia of the solid and of  the so-called ``added inertia'' which is a symmetric  positive-semidefinite  matrix depending only on the body's shape and position, and which encodes the amount of  incompressible fluid that the  rigid body  has also to accelerate
around itself. 
Remarkably enough in the case without control and where the circulation is $0$ it turns out that the solid equations can be recast as a geodesic equation associated with the metric given by the total inertia. 

Then we will extend this analysis to the case where there is a control on a part of the external boundary in Section \ref{NODEB}, cf. Theorem \ref{reform}. In particular we will establish that the remote influence of the external boundary control translates into two additional force terms in the second-order ODE for the solid position; indeed we will distinguish one force term associated with the control velocity and another one associated with its time derivative.

To simplify notations, we denote the positions and velocities
${q}=(h,\vartheta)$, $q'=(h',\vartheta')$, and
$$\mathcal{S}(q)=h+R(\vartheta)(\mathcal{S}_0-h_0) \text{ and }
\mathcal{F} (q)=\Omega\setminus\mathcal{S}(q),$$ since the dependence in time of the domain occupied by the solid comes only from the position $q$. Furthermore, we denote $q(t)=(h(t),\vartheta (t))$.

\subsection{A reminder of the uncontrolled case}
\label{NODEA}

We first recall that in the case without any control the fluid velocity satisfies (\ref{bc1}), (\ref{bc2}), (\ref{kelv}) and (\ref{helm}). 
Therefore at each time $t$ 
the fluid velocity $u$ satisfies the following div/curl system:
\begin{equation}
\label{zozo}
  \left\{
      \begin{aligned}
&\div u = \curl u= 0 \ \text{ in } \mathcal{F}(q) ,\\
&u \cdot n =0 \ \text{on}\ \partial \Omega
 \ \text{ and  } \ u \cdot n =  \left( h'+\vartheta' (x-h)^\perp \right)\cdot n\ \text{on}\ \partial \mathcal{S}(q),\\
&\int_{\partial\mathcal{S}(q)} u \cdot \tau \, d\sigma =\gamma ,
\end{aligned} \right.
\end{equation}
where the  dependence in time is only due to the one of $q$ and $q'$.
Given the solid position $q$ and 
 the right hand sides, the system \eqref{zozo} uniquely determines  the fluid velocity $u$ in the space of $C^\infty$ vector fields on the closure of $ \mathcal{F}(q)$. 
Moreover thanks to the linearity of the  system with respect to its right hand sides, its unique solution $u$ can be uniquely decomposed with respect to the following functions which depend only on the solid position  $q=(h,\vartheta)$ in $\mathcal{Q}$  and encode the contributions of elementary right hand sides.
\begin{itemize}
\item 
The Kirchhoff potentials 
\begin{equation}
  \label{phi}
 \mathbf{\Phi} =(\Phi_1,\Phi_2,\Phi_3)(q,\cdot)
\end{equation}
are defined as the solution of the Neumann problems
\begin{align} \label{kir}
\begin{split}
&\Delta\Phi_i(q,x)=0\ \text{in}\ \mathcal{F}(q), \ \ \partial_{n}\Phi_i(q,x)=0\ \text{on}\ \partial\Omega,\text{ for } i\in\{1,2,3\},\\
&\partial_{n}\Phi_i(q,x)=
\left\{ \begin{array}{l}
n_i\ \text{on}\ \partial\mathcal{S}(q),\text{ for }i\in\{1,2\},\\
(x-h)^\perp \cdot n\ \text{on}\ \partial\mathcal{S}(q),\text{ for }i=3,
\end{array} \right.
\end{split}
\end{align} 
where all differential operators are with respect to the variable $x$.
\item  The stream function ${\psi}$ for the circulation term is defined in the following way. First we consider the solution $\tilde{\psi}(q,\cdot)$ of the Dirichlet problem
$\Delta\tilde{\psi}(q,x)=0$ in $ \mathcal{F}(q)$, 
$\tilde{\psi}(q,x)=0$ on $ \partial\Omega$, 
$\tilde{\psi}(q,x)=1$ on $ \partial\mathcal{S}(q).$
Then we set 
\begin{align}\label{str}
\psi(q,\cdot)=-\left(\int_{\partial\mathcal{S}(q)} \partial_n \tilde{\psi}(q,x) \, d\sigma \right)^{-1}\tilde{\psi}(q,\cdot),
\end{align}
such that we have
$$\int_{\partial\mathcal{S}(q)} \partial_n \psi(q,x) \, d\sigma = -1,$$
noting that the strong maximum principle gives us $\partial_n \tilde{\psi}(q,x)<0$ on $\partial\mathcal{S}(q)$.
\end{itemize}
%
%
%
%
\begin{remark} \label{RemReg}
The  Kirchhoff potentials $\mathbf{\Phi}$ and the  stream function ${\psi}$  are 
 $C^\infty$  as functions of $q$ on $\mathcal{Q}$. 
We will use several times some properties of regularity with respect to the domain of solutions to linear elliptic problems, 
included for another potential $\mathcal A[q,g]$ associated with the control, see  Definition~\ref{defA} below. 
We will mention along the proof the properties which will be used and 
 we refer  to \cite{ChambrionMunnier,hp,LM} for 
 more on this material which is now standard in fluid-structure interaction. 
\end{remark}
The following statement is an immediate consequence of the definitions above.
\begin{lemma}\label{ldecomp0}
For any $q=(h,\vartheta)$ in $\mathcal{Q}$,
for any $p=(\ell,\omega)$ in $\mathbb R^2 \times \mathbb R$ and for any 
$\gamma$,   the unique solution $u$ in $C^\infty ( \overline{\mathcal{F}(q)})$ to the following system:
\begin{equation}
\label{zozoFormal}
  \left\{
      \begin{aligned}
&\div u = \curl u= 0 \ \text{ in } \in \mathcal{F}(q) ,\\
&u \cdot n =0 \ \text{on}\ \partial \Omega
  \text{ and  } u \cdot n =  \left( \ell+ \omega (x-h)^\perp \right)\cdot n\ \text{on}\ \partial \mathcal{S}(q),\\
&\int_{\partial\mathcal{S}(q)} u \cdot \tau \, d\sigma =\gamma. 
\end{aligned} \right.
\end{equation}
 is given by the following formula, for $x$ in $\overline{\mathcal{F}(q)}$, 
\begin{equation}
  \label{praud}
u(x)=\nabla (p \cdot \Phi (q,x))+\gamma \nabla^{\perp}\psi(q,x).
\end{equation}
\end{lemma}
Above $p \cdot \Phi  (q,x)$ denotes the inner product $p \cdot \Phi (q,x) = \sum_{i=1}^3 \, p_i  \Phi_i  (q,x)$.
\bigbreak
Let us now address the solid dynamics. 
The solid motion is driven by the Newton equations (\ref{newt}) where the influence of the fluid on the solid appears through the fluid pressure. 
The pressure can in turn  be related to the fluid velocity thanks to the Euler equations (\ref{eu}). 
The contributions  to the solid dynamics of the 
 two terms  in  the right hand side of the  fluid velocity decomposition formula  \eqref{praud} are very different. 
 On the one hand the potential part, i.e. the first term  in  the right hand side of    \eqref{praud}, contributes as an added inertia matrix, together with a connection term which ensures a geodesic structure (see \cite{Munnier}), whereas on the other hand the contribution of the  term due to the circulation,  i.e.  the second term  in  the right hand side of    \eqref{praud}, turns out 
  to be a force which reminds us of the Lorentz force in electromagnetism by its structure (see \cite{GMS}). 
We therefore introduce the following notations.
\begin{itemize}
\item 
 We respectively define the genuine and added mass $3 \times 3$ matrices by
$$\mathcal{M}_{g}= \left( \begin{array}{ccc}
m & 0 & 0\\
0 & m & 0\\
0 & 0 & \mathcal{J}
\end{array} \right) ,$$
and, for $q\in\mathcal{Q}$, 
$$ \mathcal{M}_{a} (q) =\left(  \int_{  \mathcal{F} (q)}   \nabla\Phi_i (q,x) \cdot \nabla\Phi_j (q,x) \, dx \right)_{1 \leqslant i,j  \leqslant 3} .
$$
Note that $\mathcal{M}_a$ is a symmetric Gram matrix and  is  $C^\infty$ on $\mathcal{Q}$. 
\item 
We define the symmetric bilinear map $\Gamma (q) $ given by
$$\langle \Gamma(q), p,p \rangle=\left( \sum_{1\leq i,j \leq 3} \Gamma^k_{i,j} (q) \,  p_i \,  p_j \,  \right)_{1\leq k \leq 3} \in \mathbb{R}^3,\ \forall p \in \mathbb{R}^3,$$
where,  for each $i,j,k\in\{1,2,3\}$,
$\Gamma^k_{i,j} $ denotes  the Christoffel symbols of the first kind defined on $\mathcal{Q}$ by 
\begin{align}\label{chrs}
\Gamma^k_{i,j}  =\frac{1}{2}\left( \frac{\partial (\mathcal{M}_a)_{k,j} }{\partial q_i}  + \frac{\partial (\mathcal{M}_a)_{k,i} }{\partial q_j} -\frac{\partial (\mathcal{M}_a)_{i,j} }{\partial q_k} \right)  .
\end{align}
It can be checked that $\Gamma$ is of class $C^\infty$ on $\mathcal{Q}$.
\item 
We introduce the following $C^\infty$  vector fields on $\mathcal{Q}$ with values in $\mathbb R^3$ by 
\begin{align}
 {E} &=-\frac{1}{2}\int_{\partial\mathcal{S}(q)}|\partial_n \psi(q,\cdot)|^2  \partial_n \Phi(q,\cdot)  \, d\sigma ,
\\ {B} &=\int_{\partial\mathcal{S}(q)} \partial_n \psi(q,\cdot) \left( \partial_n \Phi(q,\cdot) \times \partial_\tau \Phi(q,\cdot) \right)\, d\sigma .
\end{align}
\end{itemize}
We recall that the notation $\Phi$ was given in \eqref{phi}.  

The reformulation of the model as an ODE is given in the following result, which  was first established  in \cite{Munnier} in the case $\gamma=0$ and in \cite{GMS} in the case $\gamma \in \R$. 
\begin{theorem}\label{reform}
Given 
$q=(h,\vartheta) \in C^{\infty}([0,T];\mathcal{Q})$, $u\in C^{\infty}(\mathcal{D}_T;\mathbb{R}^{2})$
 we have that $(q,u)$  is a solution to 
 (\ref{eu}), (\ref{bc1}),  (\ref{bc2}), (\ref{newt}),  (\ref{kelv}) and (\ref{helm})   if and only if $q$ satisfies the following ODE on $[0,T]$
 \begin{align} \label{tout}
\Big(\mathcal{M}_g+\mathcal{M}_a(q)\Big)q''+\langle \Gamma(q),q',q'\rangle =  \gamma^2 E(q) + \gamma q'\times B(q) ,
\end{align}
and $u$ is the unique solution to 
the system (\ref{zozo}). 
Moreover the total kinetic energy $\frac12 \Big(\mathcal{M}_g+\mathcal{M}_a(q)\Big) q' \cdot q'$ is conserved in time for smooth solutions of (\ref{tout}), at least as long as there is no collision.
\end{theorem}

Note that  in the case where  $\gamma=0$, the ODE (\ref{tout}) means that the particle $q$ is moving along the geodesics associated with the Riemannian metric induced on $\mathcal{Q}$ by the matrix field $ \mathcal{M}_g+\mathcal{M}_a(q)$. 
Note that, since $\mathcal{Q}$ is a manifold with boundary and the metric $ \mathcal{M}_g+\mathcal{M}_a(q)$ 
 may become singular at the boundary of $\mathcal{Q}$, the Hopf-Rinow theorem does not apply and geodesics may not be global.
 However we will make use only of local geodesics. 

\begin{remark}
Let us also mention that the whole ``inviscid fluid + rigid body''  system can be reinterpreted as a geodesic flow on an infinite dimensional manifold, cf. \cite{GS-Geod}.
However the reformulation established by Theorem  \ref{reform}  relies on the finite dimensional manifold $\mathcal{Q}$ and sheds more light on the dynamics of the rigid body. 
 \end{remark}
 
Below we provide a sketch of the proof of Theorem \ref{reform}; this will be useful in  Section \ref{NODEB} when extending the analysis to the controlled case. 

\begin{proof}
Let us focus on the direct part of the proof for sake of clarity but all the subsequent arguments can be arranged in order to insure the converse part of the statement as well. 
Using Green's first identity and the properties of the Kirchhoff functions, the Newton equations (\ref{newt}) can be rewritten as
\begin{align}
\label{renew}
\mathcal{M}_g \, q'' =\int_{  \mathcal{F} (q)} \nabla \pi \cdot \nabla \Phi (q,x) \, dx.
\end{align}
Moreover when  $u$ is irrotational, Equation (\ref{eu}) can be rephrased as
\begin{equation}
  \label{Bernou}
\nabla \pi = -\partial_{t}u-\frac12 \nabla_{x} |u|^{2} , \quad  \text{ for }  x  \text{ in }  \mathcal{F} (q(t)),
\end{equation}
and  Lemma \ref{ldecomp0} shows that for any $t$ in $[0,T]$,  
\begin{equation}
  \label{baba}
u(t,\cdot)=\nabla (q'(t) \cdot \Phi (q(t),\cdot))+\gamma \nabla^{\perp}\psi(q(t),\cdot) .
\end{equation}
Substituting  \eqref{baba} into   \eqref{Bernou} and then the resulting decomposition of $\nabla \pi$ into \eqref{renew} we get
\begin{align}
\label{devdev}
\begin{split}
\mathcal{M}_g\, q''=
- \int_{  \mathcal{F} (q)} \left( \partial_{t} \nabla (q'\cdot \Phi(q,x)) + \frac{\nabla|\nabla (q'\cdot \Phi(q,x))|^{2}}{2}  \right) \cdot \nabla \Phi (q,x) \, dx \\
- \gamma \int_{\mathcal{F} (q)}\left(  \partial_{t} \nabla^\perp \psi(q,x) + \nabla\left( \nabla(q'\cdot\Phi (q,x))\cdot \nabla^\perp\psi(q,x)\right) \right)\, \cdot \nabla \Phi (q,x) \, dx\\ -\gamma^2\int_{ \mathcal{F} (q)} \frac{\nabla|\nabla\psi(q,x)|^{2}}{2} \, \cdot \nabla \Phi (q,x) \, dx .
 \end{split}
\end{align}
According to Lemmas 32, 33 and 34 in \cite{GMS}, 
the terms in the three lines of 
the right-hand side above are respectively equal to $-\mathcal{M}_a (q) q''-\langle \Gamma(q),q',q'\rangle$, 
$ \gamma q'\times B(q)$ and $\gamma^2 E(q)$, so that we easily deduce 
 the ODE (\ref{tout}) from \eqref{devdev}. 
 
 The conservation of  the kinetic energy $\frac12 \Big(\mathcal{M}_g+\mathcal{M}_a(q)\Big) q' \cdot q'$  is then simply obtained by multiplying 
the ODE (\ref{tout}) by $q'$ and observing that 
\begin{equation}
  \label{identityNRJ}
\Big(\Big(\mathcal{M}_g+\mathcal{M}_a(q)\Big)q''  +\langle \Gamma(q),q',q'\rangle  \Big)\cdot q'  = 
\Big(\frac12 \Big(\mathcal{M}_g+\mathcal{M}_a(q)\Big) q' \cdot q'  \Big)' .
\end{equation}
\end{proof}

\subsection{Extension to the controlled case}
\label{NODEB}

We now tackle the case where a control is imposed on the part $\Sigma$ of the external boundary $\partial \Omega$. 
At any time this control has to be compatible with the incompressibility of the fluid meaning that the flux through $\Sigma$ has to be zero. 
We therefore introduce the set 
$$ \mathcal C := \left\{  g \in    C_{0}^{\infty}( \Sigma  ;\mathbb{R}) \  \text{ such that }  \,  \int_\Sigma g \, d \sigma=0 \right\} .$$

The  decomposition of the fluid velocity  $u$ then involves a new potential term involving the following function. 
\begin{definition}\label{defA}
With  any $q\in\mathcal{Q}$ and $g \in \mathcal C$ we associate the unique solution
$\overline{\alpha} := \mathcal A[q,g] \in C^\infty (\overline{\mathcal{F}(q)};\mathbb{R})$  to the following Neumann problem:
\begin{equation}\label{pot}
\Delta \overline{\alpha} =0\ \text{in}\  \mathcal{F}(q) \quad    \text{ and } \quad 
\partial_{n} \, \overline{\alpha}=g\mathbbm{1}_{\Sigma}\ \text{on}\ \partial\mathcal{F}(q) ,
\end{equation}
with zero mean on $\mathcal{F}(q) $.
\end{definition}
Let us mention that the zero mean condition above allows to determine a unique solution to the Neumann problem but plays no role in the sequel. 

Now Lemma \ref{ldecomp0} can be modified as follows. 
\begin{lemma} \label{ldecomp0C}
For any $q=(h,\vartheta)$ in $\mathcal{Q}$,  for any $p=(\ell,\omega)$ in $\mathbb R^2 \times \mathbb R$, for any  $\overline{g}$ in $ \mathcal C$, 
 the unique solution $u$ in $C^\infty ( \overline{\mathcal{F}(q)})$ to 
\begin{align*}
&\div u = \curl u= 0 \ \text{ in }  \mathcal{F}(q) ,\\
&u \cdot n = \mathbbm{1}_\Sigma \, \overline{g}\ \text{on}\ \partial \Omega
  \text{ and  } u \cdot n =  \left(\ell +\omega (x-h)^\perp \right)\cdot n\ \text{on}\ \partial \mathcal{S}(q),\\
&\int_{\partial\mathcal{S}(q)} u \cdot \tau \, d\sigma =\gamma , 
\end{align*}
 is given by 
 \begin{equation}
  \label{decomp2}
u =\nabla (p \cdot \Phi(q,\cdot))+\gamma \nabla^{\perp}\psi(q,\cdot) + \nabla \mathcal A[q,\overline{g}] .
\end{equation}
\end{lemma}
Let us avoid a possible confusion by mentioning that the $\nabla$ operator above has to be considered with respect to the space variable $x$. 
The function $ \mathcal A[q,\overline{g}]$  and its time derivative will respectively be involved into the arguments of the following force terms.
\begin{definition}
\label{def-forces}
We define, for any $q$ in $\mathcal{Q}$, $p$ in $\mathbb{R}^3$, $\alpha $ in $C^{\infty}(\overline{\mathcal{F}(q)}; \mathbb R)$ and $\gamma$ in $\mathbb R$, 
$F_1  (q,p, \gamma) [ \alpha]  $ and $F_2 (q) [\alpha]   $ in $\mathbb R^3$ by 
\begin{align} \label{ef2}
 F_1 (q,p, \gamma) [ \alpha]  &:= -  \frac12 \int_{ \partial \mathcal{S} (q)}  |\nabla\alpha|^{2}   \, \partial_n \Phi (q,\cdot) \, d \sigma
\\  \nonumber & \quad \quad-\int_{ \partial \mathcal{S} (q)}  \nabla\alpha \cdot \Big(\nabla (p\cdot \Phi(q,\cdot))+\gamma \nabla^{\perp}\psi(q,\cdot)\Big)\, \partial_n \Phi(q,\cdot) \, d \sigma,
\\   F_2  (q) [\alpha]  \label{ef3}
&:= - \int_{ \partial \mathcal{S} (q)}   \alpha  \, \partial_n \Phi (q,\cdot) \, d \sigma .
\end{align}
\end{definition}

 Observe that Formulas \eqref{ef2} and \eqref{ef3} only require  $\alpha $ and $\nabla \alpha $ to be defined on $\partial \mathcal{S} (q)$. 
 Moreover when these formulas  are applied to 
$\alpha = \mathcal A[q,g] $  for some $g $ in $\mathcal C$, 
then only the trace of $\alpha$ and the tangential derivative $\partial_\tau \alpha$ on $\partial \mathcal{S} (q)$
are involved, since the normal derivative of $\alpha$ vanishes on $\partial \mathcal{S} (q)$ by definition, cf. \eqref{pot}.

We define our notion of controlled solution of the ``fluid+solid'' system as follows. 
\begin{definition}\label{CS}
We say that $(q,g)$ in $C^{\infty} ([0,T];\mathcal{Q}) \times C^{\infty}_0 ([0,T]; \mathcal C)$ is a controlled solution if the following ODE holds true on $[0,T]$:
  \begin{align}\label{renform}
  \begin{split}
\big( \mathcal{M}_g+\mathcal{M}_a(q) \big) q''+\langle \Gamma(q),q',q'\rangle &=  \gamma^2 E(q) + \gamma q'\times B(q) 
\\   &\quad + F_1 (q,q', \gamma) [ \alpha] +   F_2  (q) [\partial_t\alpha] ,
\end{split}
\end{align}
 where $\alpha(t,\cdot) := \mathcal A[q(t),g(t,\cdot)] $.
\end{definition}

We have the following result for reformulating the model as an ODE.
\begin{proposition} \label{reformC}
Given 
$$q\in C^{\infty}([0,T];\mathcal{Q})  ,\quad u\in C^{\infty}(\mathcal{D}_T;\mathbb{R}^{2})   \quad    \text{ and  } \quad  g\in  C^{\infty}_0 ([0,T]; \mathcal C)  ,$$
 we have that $(q,u)$
 is a solution to (\ref{eu}), (\ref{bc1}),   (\ref{newt}), (\ref{kelv}), (\ref{ic}), (\ref{helm}), (\ref{yud1}), (\ref{yud3}), (\ref{yud2}) if and only if $(q,g)$ 
 is a controlled solution
 and $u$ is the unique solution to the unique div/curl type problem:  
\begin{align*}
&\div u = \curl u= 0 \ \text{ in }  \mathcal{F}(q) ,\\
&u \cdot n =  \mathbbm{1}_\Sigma  \, g\ \text{on}\ \partial \Omega
  \text{ and  } u \cdot n =  \left( h'+\vartheta' (x-h)^\perp \right)\cdot n\ \text{on}\ \partial \mathcal{S} (q) ,\\
&\int_{\partial\mathcal{S}(q)} u \cdot \tau \, d\sigma =\gamma , 
\end{align*}
with $q=(h,\vartheta)$. 
\end{proposition}
Proposition \ref{reformC} therefore extends 
Theorem \ref{reform} to the case with an external boundary control 
(in particular one recovers Theorem \ref{reform} in the case where $g$ is identically vanishing). 
\begin{proof}
We proceed as in the proof of Theorem \ref{reform} recalled above, with some modifications due to the 
 extra term involved in the decomposition of the fluid velocity, compare \eqref{praud} and \eqref{decomp2}. 
 In particular some extra terms appear in the right hand side of \eqref{devdev}  after substituting the right hand side of \eqref{decomp2} for $u$ in 
   \eqref{Bernou}. Using some integration by parts and the properties of the Kirchhoff functions we obtain integrals on $\partial\mathcal{S}(q)$ whose sum precisely gives
   $F_1( q,q' ,\gamma) [\alpha(t,\cdot)] +   F_2  (q) [\partial_t\alpha(t,\cdot)]$. This allows to conclude. 
\end{proof}

\section{Reduction to the case where the displacement, the velocities and the circulation are small}
\label{RELO}

For $\delta>0$, we introduce the set
\begin{align}\label{qd}
Q_\delta =\{q\in\Omega\times\mathbb{R}:\ d(\mathcal{S}(q),\partial\Omega) > \delta\}.
\end{align}

The goal of this section is to prove that Theorem \ref{main} can be deduced from the following result. The balls have to be understood for the Euclidean norm (rather than for the metric $\mathcal{M}_g+\mathcal{M}_a(q)$).
\begin{theorem}
\label{lmain}
Given $\delta>0$, $\mathcal{S}_0\subset\Omega$ bounded, closed, simply connected with smooth boundary, which is not a disk,  $q_0 $ in $\mathcal{Q}_\delta$ and 
 $T>0$, there exists  $r>0$ such that for any 
 $q_1 $ in  ${B}(q_0,r)$, 
for any $\gamma\in\mathbb{R}$ with $|\gamma|\leq r$ and for any $q'_0 , q'_1\in {B}(0,r)$, 
there is a controlled solution $(q,g) $ in $C^\infty ([0,T];\mathcal{Q}_\delta) \times C_0^{\infty}([0,T]\times\Sigma)$  such that
$(q,q')(0)=(q_0,q'_0)$ and  $(q,q')(T)=(q_1,q'_1)$.
\end{theorem}
Remark in particular that for $r>0$ small enough, $B(q_{0},r)$ is included in the connected component of $\mathcal{Q}_\delta$ containing $q_0$.
\begin{proof}[Proof of Theorem \ref{main} from Theorem \ref{lmain}]

We proceed in two steps: first we use a time-rescaling argument in order to deduce from Theorem \ref{lmain} a more general result covering the case
where the initial and final velocities $q'_0 $ and $ q'_1$ and the circulation $\gamma$ are large. This argument is reminiscent of a  time-rescaling argument used by J.-M. Coron for the Euler equation \cite{Coron:EC}, which has been also used in \cite{G-R} in order to pass from the potential case to the case with vorticity. 
Then we use a compactness  argument  in order to deal with the case where $q_0 $ and $q_1$ are remote (but of course in the same  connected component of $\mathcal{Q}_\delta$). \par
The time-rescaling argument relies on the following observation: 
it follows from \eqref{renform} that $(q,g)$ is a controlled solution on $[0,T]$ with circulation $\gamma$ if and only if 
$(q^\lambda,g^\lambda)$ is a controlled solution on $[0, \lambda T ]$ with circulation $\frac{\gamma}{\lambda}$,
where $(q^\lambda,g^\lambda)$ is defined by 
\begin{equation}
  \label{sca-lambda}
q^{\lambda}(t) :={q}\left(\frac{t}{\lambda}\right)   \text{ and } g^{\lambda}(t,x) :=\frac{1}{\lambda}{g}\left(\frac{t}{\lambda},x \right) .
\end{equation}
Of course the initial and final conditions 
$$
(q,q')(0)= (q_0 ,q'_0 )  \text{ and }   (q,q' )(T)= (q_1 ,  q'_1)
$$
translate respectively into 
\begin{equation} \label{RR}
(q^{\lambda}, (q^{\lambda})')(0) = \left(q_0, \frac{q'_0}{\lambda} \right)
\text{ and } (q^{\lambda} , (q^\lambda)')(\lambda T) = \left(q_1,  \frac{q'_1}{\lambda}\right) .
\end{equation}
Now consider  $q_0 $ in $\mathcal{Q}_\delta$ and $q_1 $ in  $\overline{B}(q_0,r)$ in the same connected component of $\mathcal{Q}_\delta$ as $q_0$, with $r>0$ as in Theorem \ref{lmain}, and $q'_0 $, $ q'_1$ and  $\gamma$ without size constraint. For $\lambda$ small enough, $\left(q_0, \lambda q'_0\right)$, $\left(q_1, \lambda q'_1\right)$ and $\lambda \gamma$ satisfy the assumptions of Theorem~\ref{lmain}. Hence there exists a 
controlled solution $(q,g)$ on $[0, T ]$, achieving $(q,q')(0)=(q_0,\lambda q'_0)$ and 
$(q,q')( T)=(q_1, \lambda q'_1)$. 
On the other hand, the corresponding trajectory $q^\lambda$ constructed above will satisfy the conclusions of Theorem \ref{main} on $[0,\lambda T]$, in particular that $(q^{\lambda}, (q^{\lambda})')(0) = \left(q_0, q'_0 \right)
\text{ and } (q^{\lambda} , (q^\lambda)')(\lambda T) = \left(q_1,  q'_1\right) .$
Moreover we can assume that it is the case without loss of generality that $\lambda$ is small, and in particular that $\lambda \leq 1$. Thus the result is obtained but in a shorter time interval.

To get to the desired time interval, using that Equation  \eqref{renform}  enjoys some invariance properties by translation and time-reversal (up to the change of the sign of $\gamma$) it is sufficient to glue together an odd number, say $2N+1$ with $N$ in $\mathbb N^*$, of  appropriate controlled solutions each defined on a time interval of length $\lambda T$ with $\lambda =\frac{ 1 }{ 2N+1} $,  going 
back and forth between  $(q_0,q'_0)$ and $(q_1,q'_1)$ until time $T=(2N+1)\lambda T$. Moreover one can see that the gluings are not only $C^2$ but even $C^\infty$. 

We have therefore already proven that  Theorem \ref{main} is true 
 in the case where $q_1 $  is close to $q_0$, or more precisely for any   $q_0 $ in $\mathcal{Q}_\delta$ and 
 $q_1 $ in  $\overline{B}(q_0,r_{q_0} )$.

For the general case where $q_0$ and $q_1 $ are in the same connected component of $\mathcal{Q}_\delta$ for some $\delta>0$,  without the closeness condition, 
we use again a gluing process. 
Consider indeed a smooth curve from $q_0$ to $q_1$. For each point $q$ on this curve, there is a $r_q>0$ such that for any $\tilde{q}$ in $B(q,r_q)$, any $q'_0$, $q'_1$ and any $\gamma$, one can connect $(q,q_0')$ to $(\tilde{q},q_1')$ by a solution of the system, for any time $T>0$.
Extract a finite subcover of the curve  by the balls $B(q,r_q)$. Therefore we find
 $N \geq 2$ and $(q_{\frac{i}{N} })_{i=1,\ldots,N-1}$ in the same connected component of $\mathcal{Q}_\delta$ as $q_0$ such that 
 for any $i=1,\ldots,N$, $q_{\frac{i}{N}}$ is  in $\overline{B}(q_{\frac{i-1}{N}} ,r_{q_\frac{i-1}{N}} )$ (note that this includes $q_0$ and $q_1$). 
Therefore, using again the local result obtained above, there exist some controlled solutions from $( q_{\frac{i-1}{N}} , 0 )$ to $(q_{\frac{i}{N}} , 0 )$ (for $i=1$ and $i=N$ we use $(q_0,\frac{q'_0}{N})$ and $(q_1,\frac{q'_1}{N})$ rather than $(q_0,0)$ and $(q_1,0)$), each on a  time interval  of length $T$ associated with circulation $\frac{\gamma}{N} $.  One deduces by time-rescaling some controlled solutions associated with circulation $\gamma$ on a time interval of length $\frac{ T }{ N}$.  
Gluing them together leads to the desired controlled solution.
\end{proof}
%
%
%
%
%
%
%
%
%
%
%
%
\section{Reduction to an approximate controllability result}
\label{REAP}
The goal of this section is to prove that Theorem~\ref{lmain} can be deduced from the following approximate controllability result thanks to a topological argument already used in \cite{G-R}, see Lemma~\ref{top} below.
Let us mention that a similar argument has also  been used for control purposes but in other contexts, see e.g.
 \cite{Aronsson,BL,Grasse1,Grasse2}. 
\begin{theorem} \label{approx}
Given $\delta>0$, $\mathcal{S}_0\subset\Omega$ bounded, closed, simply connected with smooth boundary,
which is not a disk, $q_0$ in $\mathcal{Q}_\delta$ and $T>0$,
there is $\tilde{r}>0$ such that $B(q_0 , \tilde{r}) $ is included in the same connected component of $\mathcal{Q}_\delta$ as $q_0$, and furthermore,
for any $\eta >0$, there exists $r'=r'(\eta)>0$ such that
for any $\gamma\in\mathbb{R}$ with $|\gamma|\leq r'$ and for any $q'_0 $ in $\overline{B}(0,\tilde{r})$, 
there is a mapping
$$
\mathcal{T}: \overline{B}\big( (q_0 , q'_0), \tilde{r} \big) \rightarrow C^\infty ([0,T]; \mathcal{Q}_\delta)
$$
which with $ ({q}_1 , {q}'_1) $ associates $q$ where $(q,g) $ is a controlled solution associated with the initial data $(q_0 , q'_0)$,
such that the 
 mapping
$$({q}_1 , {q}'_1)  \in \overline{B}\big( (q_0 , q'_0), \tilde{r} \big) \mapsto \big(\mathcal{T}({q}_1 , {q}'_1 )  , \mathcal{T}({q}_1 , {q}'_1 )'   \big) (T)   \in \mathcal{Q}_\delta \times \R^3 $$
is continuous and such that 
 for any $ ({q}_1 , {q}'_1) $ in $ \overline{B}\big( (q_0 , q'_0), \tilde{r} \big)$,
\begin{equation*}
\|  \big(\mathcal{T}({q}_1 , {q}'_1 )  , \mathcal{T}({q}_1 , {q}'_1 )'   \big) (T)    -( {q}_1 , {q}'_1 ) \| \leqslant \eta .
\end{equation*}
\end{theorem}
The proof of Theorem \ref{approx} will be given in Section \ref{proof-approx}. 
Here we prove that Theorem \ref{lmain} follows from Theorem \ref{approx}.
\begin{proof}[Proof of Theorem \ref{lmain} from Theorem \ref{approx}]
Let $\delta>0$, $\mathcal{S}_0\subset\Omega$ bounded, closed, simply connected with smooth boundary, which is not a disk, $q_0 $ in $\mathcal{Q}_\delta$ and 
 $T>0$. 
 Let $\tilde{r}>0$ as in Theorem \ref{approx} and $\eta= \frac{ \tilde{r} }{ 2} $. We deduce that 
for any $\gamma\in\mathbb{R}$ with $|\gamma|\leq r'=r'( \frac{ \tilde{r} }{ 2})$ and
 $q'_0$ in $ \overline{B}(0,\tilde{r})$,
there is a mapping $\mathcal{T}: \overline{B}\big( (q_0 , q'_0), \tilde{r} \big) \rightarrow C^\infty ([0,T]; \mathcal{Q}_\delta ) $ which maps 
$ ( {q}_1 , {q}'_1 ) $ to $q$ where $(q,g) $ is a controlled solution
associated with the initial data $(q_0 , q'_0)$, such that for any  $ ( {q}_1 , {q}'_1 ) $ in $ \overline{B}\big( (q_0 , q'_0), \tilde{r} \big)$,
$ \| \big(\mathcal{T}({q}_1 , {q}'_1 )  , \mathcal{T}({q}_1 , {q}'_1 )'   \big) (T)  -( {q}_1 , {q}'_1 ) \| \leqslant  \frac{ \tilde{r} }{ 2} .$
We define a mapping 
$f $ from $ \overline{B}\big( (q_0 , q'_0), \tilde{r} \big)$ to $\mathbb R^6$ which maps $({q}_1 , {q}'_1 )$ to 
$    f ({q}_1 , {q}'_1 ) := \big(\mathcal{T}({q}_1 , {q}'_1 )  , \mathcal{T}({q}_1 , {q}'_1 )'   \big) (T)   $. 
Then we apply the following lemma borrowed from \cite[pages 32-33]{G-R}, to
 $w_0 = ({q}_0 , {q}'_0 ) $ and $ \kappa = \tilde{r} $.
\begin{lemma}\label{top}
Let $w_0 \in\mathbb{R}^n,$ $\kappa>0$, $f:\overline{B}(w_0,\kappa)\to\mathbb{R}^n$ a continuous map such that we have
$|f(w)-w|\leq \frac{\kappa}{2}$ for any $x$ in $\partial B(w_0,\kappa).$
Then
$B(w_0,\frac{\kappa}{2})\subset f(\overline{B}(w_0,\kappa)).$
\end{lemma}
This allows to conclude the proof of Theorem \ref{lmain} by setting $r=\min\left\{\frac{\tilde{r}}{2\sqrt{5}},r'(\frac{\tilde{r}}{2})\right\}$, since the conditions 
 $q_1 \in {B}(q_0,r)$, 
 $|\gamma|\leq r$ and $q'_0 , q'_1\in {B}(0,r)$ imply $|\gamma|\leq r'( \frac{ \tilde{r} }{ 2})$
 and $(q_1,q'_1)\in B((q_0,q'_0), \frac{ \tilde{r} }{ 2} )$.
 \end{proof}

%
%
%
%
%
\section{Proof of the approximate controllability result Theorem~\ref{approx}}
\label{proof-approx}
In this section we prove Theorem \ref{approx} by  exploiting the geodesic feature of the uncontrolled system with zero circulation, cf. the observation below Theorem \ref{reform}.
To do so, we will use some well-chosen impulsive controls which allow to modify  the velocity $q'$ in a short time interval and put the state of the system on a prescribed geodesic (and use that $|\gamma|$ is small). 
We mention here \cite{Br} and the references therein for many more examples on the impulsive control strategy. 
\subsection{First step}
\label{firststep}
We consider $\mathcal{S}_0\subset\Omega$ as before and consider $\delta>0$ so that $q_0 \in \mathcal{Q}_\delta$. We let $r_1 > 0$ be small enough so that  $B(q_0 , r_1) \subset \mathcal{Q}_\delta$.
We also let $T>0$. \par

The first step consists in considering the geodesics associated with the uncontrolled, potential case ($\gamma=0$).
The following classical result regarding the existence of geodesics can be found for instance in \cite[Section 7.5]{MR}, see also  \cite{Gai} for the continuity feature. 
\begin{lemma} \label{geo}
There exists  $r_2 $ in $(0,\frac12 r_1)$  such that for any $q_1 $ in $ \overline{B}(q_0 , r_2)$ there exists a unique $C^\infty$ solution $\bar{q}(t)$ lying in  $B(q_0 ,\frac12 r_1 )$ to 
\begin{align}\label{geod}
\begin{split}
\big( \mathcal{M}_g+\mathcal{M}_a(\bar{q}) \big)\bar{q}''+\langle\Gamma (\bar{q}),\bar{q}',\bar{q}'\rangle=0  \text{  on } [0,T], 
  \text{ with  } \bar{q}(0)= q_0,\ \bar{q}(T)={q}_1.
\end{split}
\end{align}
Furthermore the map
$q_1 \in \overline{B}(q_{0},r_{2}) \mapsto(c_0,c_1)\in\mathbb{R}^6$
given by $c_0=\bar{q}'(0),\ c_1= \bar{q}'(T)$ is continuous.
\end{lemma}
Let us fix $r_2 $ as in the lemma before.
Let $q'_0 $ in $\overline{B}(0,r_2)$ and  $ ( {q}_1 , {q}'_1 ) $ in $ \overline{B}\big( (q_0 , q'_0), r_2  \big)$.
\bigbreak
Our goal is to make the system follow approximately such a geodesic $\overline{q}$ which we consider fixed during this Section. For the geodesic equation in \eqref{geod},
$q_{0}$ and $q_{1}$ determine the initial and final velocities (which of course differ in general from $q_{0}'$ and $q_{1}'$).
But we will see that is possible to use the penultimate term of  \eqref{renform} in order to modify the initial and final velocities of the system.
Precisely, the control will be used so that the right hand side of \eqref{renform} behaves like two Dirac masses at time close to $0$ and $T$, driving the velocity $q'$ from the initial and final velocities to the ones of the geodesic in two short time intervals close to $0$ and $T$. 
\subsection{Illustration of the method on a toy model}
\label{Illu}
Let us illustrate this strategy on a toy model. 
We will later on adapt the analysis to the complete model, cf. Proposition \ref{apr1}.

Let  $\beta:\mathbb{R}\to\mathbb{R}$ be a smooth, non-negative function supported in $[-1,1]$, such that  $\int_{-1}^{1}\beta(t)^2 \, dt=1$ and, for $\eps$ in $(0,1)$,  
$\beta_{\varepsilon}(t) :=\frac{1}{\sqrt{\varepsilon}}\beta\left(\frac{t-\varepsilon}{\varepsilon}\right) $, 
 so that\footnote{In the next lemma we are going to make use only of the square function $\beta_{\varepsilon}^2$ but we will also have to deal with the function $\beta_{\varepsilon}$ itself 
in the sequel, see below Proposition \ref{farcontr}.}   $( \beta_{\varepsilon}^2 )_\varepsilon $  is an approximation of the unity when $\varepsilon \to 0^{+}$.

For a function $f$ defined on $[ 0,T] $, we will denote 
\begin{equation}
  \label{NormBonnasse}
 \| f \|_{T,\varepsilon} :=  \| f \|_{ C^{0}( [ 0,T] )} +   \| f \|_{C^{1}( [ 2\varepsilon,T-2\varepsilon])} .
\end{equation}
\begin{lemma} \label{toy}
Let   $q_0 $,  $r_2 $, $q_1 $, $q'_0$ and  $q'_1$ as above.
Let 
\begin{equation}
  \label{defv}
v_0 := \big( \mathcal{M}_g+\mathcal{M}_a(q_0) \big) (c_0 ( q_1) - q'_0 )
  \text{ and } v_1 := - \big( \mathcal{M}_g+\mathcal{M}_a(q_1) \big) (c_1(q_1) - q'_1 ) .
\end{equation}
Let, for $\eps$ in $(0,1)$, 
$q_{\varepsilon} $ the maximal solution to the following Cauchy problem: 
\begin{equation} \label{RER}
\big( \mathcal{M}_g+\mathcal{M}_a(q_{\varepsilon}) \big) q_{\varepsilon}''+\langle\Gamma (q_{\varepsilon}),q_{\varepsilon}',q_{\varepsilon}'\rangle =   \beta_{\varepsilon}^2 (\cdot)\,  v_0  +  \beta_{\varepsilon}^2 (T-\cdot)  v_1,
\end{equation}
with  $ q_{\varepsilon}(0)= q_0$ and $ q'_{\varepsilon}(0)={q}'_0$.
Then for $\varepsilon$ small enough, 
 $q_{\varepsilon} (t)$ lies in $ B(q_0 , r_1) $ for $t$ in $[ 0,T] $
and, 
as $\varepsilon \rightarrow 0^+$,  $\| q_{\varepsilon} - \bar{q}\|_{T, \varepsilon}  \rightarrow  0$ and  $( q_{\varepsilon} , q'_{\varepsilon} )  (T) \rightarrow  ( q_1  ,q'_1)$. 
\end{lemma}
\begin{proof}
For  $\eps$ in $(0,1)$,  let us denote $T_{\varepsilon} = \sup\,  \{  \hat{T} > 0 \text{ such that }  \,  q_{\varepsilon} (t) \in B(q_0 , r_1)   \text{ for } t \in (0,\hat{T}) \}$.  
Let us first prove that there exists $\tilde{T} > 0$ such that for any $\eps$ in $(0,1)$,  $T_{\varepsilon} \geq \tilde{T}$. 
Using the identity  \eqref{identityNRJ}, we obtain indeed, for any $\eps$ in $(0,1)$, for any $ t \in (0,T_{\varepsilon} )$, 
\begin{equation*}
 \Big(\mathcal{M}_g+\mathcal{M}_a(q_{\varepsilon} (t))\Big) q_{\varepsilon}' (t) \cdot q_{\varepsilon}'  (t)  = 
 \Big(\mathcal{M}_g+\mathcal{M}_a( q_0)\Big) {q}'_0 \cdot  {q}'_0 
 +  2 \int_0^{t} \big( \beta_{\varepsilon}^2 (\cdot)\,  v_0  +  \beta_{\varepsilon}^2 (T-\cdot)  v_1 \big) \cdot q_{\varepsilon}'   ,
\end{equation*}
Moreover, relying on Remark~\ref{RemReg}, we see that there exists $c > 0$ (which depends on $\delta$) such that for any $q$ in $\mathcal{Q}_\delta$, for any $p$ in $\R^3$, 
\begin{equation} \label{identityNRJeps}
c  |p |^2 \leq  \Big(\mathcal{M}_g+\mathcal{M}_a(q)\Big) p \cdot p \leq c^{-1}  |p |^2 .
\end{equation}
Therefore using Gronwall's lemma we obtain that there exists $C>0$ such that for any $\eps$ in $(0,1)$, for any $ t \in (0,T_{\varepsilon} )$, 
$ \sup_{t \in (0,T_{\varepsilon} )  }\,  \| q_{\varepsilon}' (t) \| \leq C .$ Therefore by the mean value theorem 
 for  $\tilde{T} := r_1 / 2C$, one has  for any $\eps$ in $(0,1)$,  $T_{\varepsilon} \geq \tilde{T}$. 

We now prove in the same time that for $\varepsilon>0$ small enough,  $T_{\varepsilon} \geq T$,  and the convergence results stated in 
Lemma \ref{toy}. 
In order to exploit the supports of the functions  $\beta_{\varepsilon} (\cdot)$ and $ \beta_{\varepsilon} (T-\cdot)$ in the right hand side of  the equation \eqref{RER} 
we compare the dynamics of $q_{\varepsilon}$ and $\bar{q}$  during the three time intervals 
 $[0, 2\varepsilon]$,  $[ 2\varepsilon,T-2\varepsilon]$ and $[ T-2\varepsilon , T]$. 

For  $ \eps_1 := \tilde{T} / 2 $ and $\eps$ in $(0,\eps_1)$,  one already has that  $T_{\varepsilon} \geq  2 \varepsilon$ and we can therefore simply compare the dynamics of 
$q_{\varepsilon}$ and $\bar{q}$ on the first interval $[0, 2\varepsilon]$.
Indeed using again  the mean value theorem we obtain that   $\sup_{t \in [0, 2\varepsilon]} \,  | q_{\varepsilon} - q_0  | $ converges to $0$ as $\varepsilon$ goes to $0$.
Moreover integrating the equation \eqref{RER} on $[0, 2\varepsilon]$ and taking into  account the choice of $v_0$ in \eqref{defv},
we obtain 
\begin{multline} \label{Grrr}
 \Big(\mathcal{M}_g+\mathcal{M}_a( q_{\varepsilon} (2\varepsilon))\Big)  q'_{\varepsilon}(2\varepsilon) 
=  \Big( \mathcal{M}_g+\mathcal{M}_a( q_0)\Big)  c_0 ( q_1) 
\\ 
-  \int_0^{2\varepsilon} \Big( D\mathcal{M}_a( q_{\varepsilon}) \cdot   q'_{\varepsilon}  \Big) \cdot q'_{\varepsilon} \, dt
- \int_0^{2\varepsilon} \langle\Gamma (q_{\varepsilon}),q'_{\varepsilon},q'_{\varepsilon}\rangle \, dt ,
\end{multline}
Now, there exists $C> 0$ such that for any $q$ in $\mathcal{Q}_\delta$, for any $p$ in $\R^3$, 
\begin{equation}
|  \big( D\mathcal{M}_a( q) \cdot   p \big) \cdot p | + | \langle\Gamma (q),p,p\rangle | \leq  C  |p |^2 .
\end{equation}
Combining this and the bound on $q_{\varepsilon}'$ we see that  the two  terms of the last line of \eqref{Grrr} above converge to $0$ as $\varepsilon$ goes to $0$.
  Since $q\mapsto \mathcal{M}_a( q)$ is continuous on  $\mathcal{Q}_\delta$ and $ q_\varepsilon (2\varepsilon) $ converges to $q_0$ as $ \varepsilon \rightarrow 0$, 
the matrix  $\mathcal{M}_a( q_{\varepsilon} )$ converges to $\mathcal{M}_a ( q_0) $  as $ \varepsilon \rightarrow 0$. 
Therefore, using that the matrix $ \mathcal{M}_g+\mathcal{M}_a( q_0)$ is invertible 
we deduce that $ q_{\varepsilon}'(2\varepsilon) $  converges to $c_0 (q_1)$ as $\varepsilon$ goes to $0$.

During the time interval $[ 2\varepsilon,T-2\varepsilon]$, the right hand side of the equation \eqref{RER} vanishes 
and the equation therefore reduces to the geodesic equation in \eqref{geod}. 
Since this equation is invariant by translation in time, one may use the following elementary result on the continuous dependence on the data, with a time shift of $2\varepsilon$.
\begin{lemma} \label{depco}
There exists $\eta >0 $  such that for any $(\tilde{q}_0, \tilde{q}'_0) $ in $B((q_0 , c_0 (q_1)) , \eta)$ there exists a unique $C^\infty$ solution $\tilde{q}(t)$ lying in  $B(q_0 , r_1 )$ to 
$\big( \mathcal{M}_g+\mathcal{M}_a(\tilde{q}) \big)\tilde{q}''+\langle\Gamma (\tilde{q}),\tilde{q}',\tilde{q}'\rangle=0$   on $ [0,T]$, 
with  $ \tilde{q}(0)= \tilde{q}_0 ,\ \tilde{q}' (0)= \tilde{q}'_0  $.
Furthermore  $\|  \tilde{q} - \overline{q}   \|_{ C^{1}( [ 0,T] )}  \rightarrow 0$ as $(\tilde{q}_0, \tilde{q}'_0) \rightarrow (q_0 , c_0 (q_1) )$.
\end{lemma}
Since $q_{\varepsilon} (2 \varepsilon)$ and $q_{\varepsilon}' (2 \varepsilon)$ respectively converge to $q_0$ and $ c_0 (q_1)$, according to Lemma~\ref{depco}
 there exists $\eps_2 $  in $(0,\eps_1)$ such that 
for  $\eps$ in $(0,\eps_2)$, there exists a unique $C^\infty$ solution $\tilde{q}_\varepsilon (t)$  
lying in  $B(q_0 , r_1 )$ to 
$\big( \mathcal{M}_g+\mathcal{M}_a(\tilde{q}_\varepsilon) \big)\tilde{q}_\varepsilon''+\langle\Gamma (\tilde{q}_\varepsilon),\tilde{q}_\varepsilon',\tilde{q}_\varepsilon'\rangle=0$   on $ [0,T]$, 
with  $ \tilde{q}_\varepsilon(0)= q_{\varepsilon} (2 \varepsilon)  ,\ \tilde{q}_\varepsilon' (0)= q_{\varepsilon}' (2 \varepsilon)  $ and
$\|  \tilde{q}_\varepsilon - \overline{q}   \|_{ C^{1}( [ 0,T] )}  \rightarrow 0$ as $ \varepsilon \rightarrow  0$.

Since the function defined by $  \hat q_{\varepsilon} (t) = q_{\varepsilon} (t + 2\varepsilon)$ also satisfies  
$\big( \mathcal{M}_g+\mathcal{M}_a(\hat q_\varepsilon) \big)\hat q_\varepsilon''+\langle\Gamma (\hat q_\varepsilon),\hat q_\varepsilon',\hat q_\varepsilon'\rangle=0$   on $ [0,T-4\varepsilon]$, 
with  $ \hat q_\varepsilon(0)= q_{\varepsilon} (2 \varepsilon)  ,\ \hat q_\varepsilon' (0)= q_{\varepsilon}' (2 \varepsilon)  $, by the uniqueness part in the Cauchy-Lipschitz theorem  one has that
$T_{\varepsilon} \geq T-  2 \varepsilon$ and 
 $\hat q_{\varepsilon}$ and $\tilde{q}_\varepsilon $
coincide on $ [0,T-4\varepsilon]$, so that, shifting back in time, 
$\| q_\varepsilon - \overline{q} (\cdot - 2\varepsilon)  \|_{ C^{1}( [  2\varepsilon,T-2\varepsilon ])}  \rightarrow 0$ as $ \varepsilon \rightarrow  0$.
Since $\overline{q}$ is smooth, this entails that $\| q_\varepsilon  - \overline{q}   \|_{ C^{1}(  [  2\varepsilon,T-2\varepsilon ])}  \rightarrow 0$ as $ \varepsilon \rightarrow  0$.

Finally one deals with the time interval $[ T-2\varepsilon , T]$ in the same way as the first step.
In particular, reducing $\varepsilon$ one more time if necessary one obtains, by an energy estimate, a Gronwall estimate and the mean value theorem, that $T_{\varepsilon} \geq T$.
Moreover
the choice of the vector $v_1$ in \eqref{defv} allows to reorientate the velocity $ q_{\varepsilon}'$ from $c_1 (q_1)$ to $q'_1$ whereas the position is not much changed (due to the uniform bound of $q'_{\varepsilon}$ and the mean value theorem) so that the value of $q_{\varepsilon} $ at time $T$ converges to $ q_1$ as $\varepsilon$ goes to $0$.
\end{proof}
\subsection{Back to the complete model}
  \label{REM}
Now in order to mimic the right  hand side of \eqref{RER}   we are going to use one part of the force term $F_1$  introduced in Definition \ref{def-forces}. 
Let us therefore  introduce some notations for the different contributions of the force term $F_1$. 
We define, for any $q$ in $\mathcal{Q}$, $p$ in $\mathbb{R}^3$, $\alpha $ in $C^{\infty}(\overline{\mathcal{F}(q)}; \mathbb R)$, 
\begin{align} 
\label{f1a} F_{1,a} (q) [ \alpha]  &:= -  \frac12 \int_{ \partial \mathcal{S} (q)}  |\nabla\alpha|^{2}   \, \partial_n \Phi (q,\cdot) \, d \sigma , 
\\   \label{f1b} F_{1,b} (q,p) [ \alpha]   &:= - \int_{ \partial \mathcal{S} (q)}  \nabla\alpha \cdot \nabla (p\cdot \Phi(q,\cdot))    \, \partial_n \Phi (q,\cdot) \, d \sigma , 
\\ \label{f1c}  F_{1,c}  (q) [\alpha]  &:=  -  \int_{ \partial \mathcal{S} (q)}  \nabla\alpha \cdot   \nabla^{\perp}\psi(q,\cdot) \, \partial_n \Phi(q,\cdot) \, d \sigma ,
\end{align}
so that for any $\gamma$ in $\mathbb R$, 
$$
F_1  (q,p, \gamma) [ \alpha]  = F_{1,a} (q) [ \alpha] +  F_{1,b} (q,p) [ \alpha]    + \gamma F_{1,c}  (q) [\alpha] .
$$
The part which will allow us to approximate the right hand side of  \eqref{RER}  is $F_{1,a} $. 
More precisely we are going to see (cf. Proposition \ref{apmain}) that there exists a control $\alpha$ (chosen below as $ \alpha = \mathcal  A [q ,g_{\varepsilon} ]$ with   $g_{\varepsilon}$  given by  \eqref{gcontrol})
such that in the appropriate regime the dynamics of \eqref{renform} behaves like the equation with  only $F_{1,a} $ on the right hand side.
Moreover the following lemma, where the time parameter does not appear,  proves that the operator $F_{1,a}  (q) [ \cdot ]  $ can actually attain any value  $v$ in $\mathbb R^3$. 
Recall that $\delta >0$ has been fixed at the beginning of Section \ref{firststep}.
\begin{proposition} \label{farcontr}
There exists a continuous mapping $\overline{g}:\mathcal{Q}_\delta \times \mathbb R^3 \rightarrow \mathcal C$ 
such that for any $(q,v) $ in $\mathcal{Q}_\delta \times \mathbb R^3$ 
the function $\overline{\alpha} := \mathcal A [q,\overline{g} (q,v)]$ in $C^\infty (\overline{\mathcal{F}(q)};\mathbb{R})$  satisfies:
\begin{gather}
\label{harmR}
\Delta \overline{\alpha}=0   \text{ in } \mathcal{F}(q),
\text{ and } \partial_{n} \overline{\alpha} = 0  \text{ on } \partial \mathcal{F}(q) \setminus \Sigma, \\
\label{hszR}
\int_{ \partial \mathcal{S} (q)}  |  \nabla\overline{\alpha}|^{2} \, \partial_n \Phi(q,\cdot) \, d \sigma = v , \\
\label{csR}
\int_{\partial\mathcal{S}(q)} \overline{\alpha} \, \partial_n \Phi(q,\cdot) \, d\sigma = 0 .
\end{gather}
\end{proposition}
We recall that the operator $\mathcal A$ was introduced in Definition~\ref{defA}. The result above will be proved in Section~\ref{S4}.
Note that when $\mathcal{S}(q)$ is a homogeneous disk, an adapted version of Proposition~\ref{farcontr} still holds, see Proposition~\ref{dfarcontr} in Section \ref{S4}.
The condition \eqref{csR} will be useful to cancel out the last term of  \eqref{renform}.
\bigbreak

We define 
\begin{equation} \label{gcontrol}
g_{\varepsilon} (t,x) :=  \beta_{\varepsilon}(t) \overline{g}(q_0, -2 v_0) (x)  +  \beta_{\varepsilon} (T-t) \overline{g}(q_1, -2 v_1)(x) ,
\end{equation}
where $v_0 $ and $v_1$ defined in \eqref{defv}, for $ ( {q}_1 , {q}'_1 ) $  in $ \overline{B}\big( (q_0 , q'_0), r_2  \big)$, and $\overline{g}$ is given by Proposition~\ref{farcontr}.
The goal is to prove that for $\varepsilon$ and $|\gamma|$ small enough, this control drives the system \eqref{renform}  with $ \alpha = \mathcal  A [q ,g_{\varepsilon} ]$ 
 from $(q_{0},q'_{0})$ to $(q_{1},q'_{1})$, approximately.
 
\ \par
\noindent
{\bf 1.} We first observe that 
\begin{multline}  \label{vendredi}
F_{1,a} (q) [\mathcal  A [q ,g_{\varepsilon} ]] = 
\beta_{\varepsilon}^2 (t) F_{1,a} (q) \Big[ \mathcal  A [q ,\overline{g} (q_0, -2v_0)  ] \Big] \\ 
+ \beta_{\varepsilon}^2 (T-t) F_{1,a} (q) \Big[ \mathcal  A [q ,\overline{g} (q_1,-2v_1)   ] \Big] ,
\end{multline}
and is therefore a good candidate to approximate the right  hand side of \eqref{RER} if $q$ is near $q_0$ for $t$ near $0$ 
and if $q$ is near $q_1$ for $t$ near $T$. One then may indeed expect that
\begin{multline*}
F_{1,a} (q) \Big[ \mathcal  A [q ,\overline{g} (q_0, -2v_0)  ] \Big]
\text{ and }
F_{1,a} (q) \Big[ \mathcal  A [q ,\overline{g} (q_1,-2v_1)   ] \Big]
\text{  are close to  } \\
F_{1,a} (q_0) \Big[ \mathcal  A [q_0 ,\overline{g} (q_0, -2v_0)  ] \Big]
\text{ and }
F_{1,a} (q_1) \Big[ \mathcal  A [q_1 ,\overline{g} (q_1,-2v_1)   ]  \Big] ,
\text{ respectively,}
\end{multline*}
 on the respective supports of $\beta_{\varepsilon} (\cdot)$ and $\beta_{\varepsilon} (T-\cdot)$.
Moreover, according to Proposition~\ref{farcontr} these last two terms are equal to $v_0$ and $v_1$ (see \eqref{f1a} and \eqref{hszR}). \par
\ \par
\noindent
{\bf 2.} Next we will rigorously prove in Proposition \ref{apr1} below that 
the conclusion of Lemma~\ref{toy} for the toy system also holds when one substitutes the term $F_{1,a} (q) [\mathcal  A [q ,g_{\varepsilon} ]]$ in \eqref{vendredi}. This corresponds also to \eqref{renform} with $\gamma=0$ and the term $F_{1,b}$ and $F_{2}$ put to zero. \par
%
\ \par
\noindent
{\bf 3.} Finally it will appear that in an appropriate regime, in particular for small $\varepsilon$ and $|\gamma|$, 
the second last  term of  \eqref{renform}  is dominant with respect to the other terms of the right hand side (here the condition 
\eqref{csR} above will be essential in order to deal with the last term of \eqref{renform}). \par
\ \par
Let us state a proposition summarizing the claims above. 
According to the Cauchy-Lipschitz theorem there exists a controlled solution $q_{\varepsilon,\gamma}$ associated
with the control $g_{\varepsilon}$ introduced in \eqref{gcontrol}, starting with the initial condition $q_{\varepsilon,\gamma}(0) = q_0$ and $q_{\varepsilon,\gamma}'(0) = q'_0$, with circulation $\gamma$, and lying in   $B(q_0 , r_1) $ up to some positive time $T_{\varepsilon,\gamma}$.
More explicitly $q_{\varepsilon,\gamma}$ satisfies on  $[0,T_{\varepsilon,\gamma}  ] $,
\begin{multline} \label{renformEPS}
\big( \mathcal{M}_g + \mathcal{M}_a(q_{\varepsilon,\gamma}) \big) q_{\varepsilon,\gamma}''
+ \langle \Gamma(q_{\varepsilon,\gamma}),q_{\varepsilon,\gamma}',q_{\varepsilon,\gamma}'\rangle
= \gamma^2 E(q_{\varepsilon,\gamma}) + \gamma q_{\varepsilon,\gamma}'\times B(q_{\varepsilon,\gamma})  \\
+ F_1 (q_{\varepsilon,\gamma},q_{\varepsilon,\gamma}', \gamma) \big[   \mathcal A[q_{\varepsilon,\gamma} ,g_{\varepsilon} ]  \big]
+ F_2 (q_{\varepsilon,\gamma}) \big[\partial_t  \mathcal A[q_{\varepsilon,\gamma} ,g_{\varepsilon} ]  \big] .
\end{multline}
 Observe that due to the choice of the control $g_{\varepsilon}$ in 
 \eqref{gcontrol} the function $q_{\varepsilon,\gamma}$ also depends on 
  $ ( {q}_1 , {q}'_1 ) $  through 
  $v_0 $ and $v_1$, see their definition in \eqref{defv}.

We have the following approximation result.
\begin{proposition} \label{apmain}
For $\varepsilon$ and $\vert \gamma \vert$  small enough,  $T_{\varepsilon,\gamma} \geqslant T$ and, 
as $\varepsilon $ and $\vert \gamma \vert$  converge to $0^+$, 
 $\| q_{\varepsilon,\gamma}-\bar{q}\|_{T, \varepsilon} \rightarrow  0$ and 
 $( q_{\varepsilon,\gamma} , q'_{\varepsilon,\gamma} )  (T) \rightarrow  ( q_1  ,q'_1)$, uniformly for  $ ( {q}_1 , {q}'_1 ) $ in $ \overline{B}\big( (q_0 , q'_0), r_2  \big)$.
\end{proposition}
This result will be proved in Section \ref{S3}. %
%
%
Once Proposition \ref{apmain} is proved, Theorem~\ref{approx} follows rapidly. Indeed, let us set $\tilde{r}=r_2$,  according to  Proposition \ref{apmain},
for $\eta >0$, there exists  $\varepsilon=\varepsilon(\eta)>0$ and $r'=r'(\eta)$ in $(0,\tilde{r})$ such that  
for any $\gamma \in \mathbb{R}$ with $|\gamma| \leq r'$ and for any $q'_0 $ in $\overline{B}(0,\tilde{r})$   
the mapping  $\mathcal{T}$ defined on $ \overline{B}\big( (q_0 , q'_0), \tilde{r} \big)$ by setting 
$\mathcal{T}({q}_1 , {q}'_1 ) = q_{\varepsilon,\gamma} $, has the desired properties. In particular the continuity of $\mathcal{T}$ follows from the regularity of $c_{0}$ in Lemma~\ref{geo} and of the solution of ODEs on their initial data.
This ends the proof of Theorem~\ref{approx}.

\subsection{About Remark~\ref{smallflux}}
  \label{REM-smallflux}
  
  Now that we presented the scheme of proof of Theorem \ref{main} let us explain how to obtain the improvement mentioned in Remark  \ref{smallflux}.
  It is actually a direct consequence of the explicit formula for $g_{\varepsilon} (t,x)$ given in \eqref{gcontrol} and of a change of variable in time. 
 Due to the expression of $\beta_{\varepsilon}$ given at the beginning of Section \ref{Illu} one obtains that the total flux through $\Sigma^{-}$, that is 
  $\int_0^T \int_{\Sigma^{-}} \, g_{\varepsilon}  \, \, d\sigma  dt$, 
  is of order  $\sqrt{\varepsilon}$. Hence one can reduce $\varepsilon$ again in order to satisfy the requirement of Remark~\ref{smallflux}.

  On the other hand observe that the time-rescaling argument used in the proof of Theorem \ref{main} from Theorem \ref{lmain}, cf.~\eqref{sca-lambda}, leaves the total flux through $\Sigma^{-}$ invariant, while 
  the number $N$ of steps involved in the end of the same proof  does not depend on $\varepsilon$. 

%
%
%
%
%
%
%
%
%
%
\section{Closeness of the controlled system to the geodesic. Proof of Proposition \ref{apmain}}
\label{S3}

In this section, we prove Proposition \ref{apmain}.

\subsection{Proof of Proposition \ref{apmain}}
\label{SS33}
The proof of Proposition \ref{apmain} is split in several parts.
To compare $q_{\varepsilon,\gamma}$ and $\overline{q}$, we are going to consider an ``intermediate trajectory'' $\tilde{q}_{\varepsilon}$ which imitates the trajectory $q_{\varepsilon}$ of the toy model of Lemma~\ref{toy}, by using the part $F_{1,a}$ of the force term.
%
More precisely we define $\tilde{q}_\varepsilon$  by
\begin{multline} \label{tildeq}
\Big(\mathcal{M}_g+\mathcal{M}_a( \tilde{q}_{\varepsilon})\Big)  \tilde{q}''_{\varepsilon}
+ \langle\Gamma (\tilde{q}_{\varepsilon}),\tilde{q}'_{\varepsilon},\tilde{q}'_{\varepsilon}\rangle
= F_{1,a} (\tilde{q}_{\varepsilon}) \big[  \mathcal A[\tilde{q}_{\varepsilon},g_{\varepsilon}] \big]  ,  \\
\text{ with }
\tilde{q}_{\varepsilon}(0)=q_0,\ \tilde{q}_{\varepsilon}'(0)=q'_0,
\end{multline}
where $g_{\varepsilon}$ was defined in \eqref{gcontrol}  and where the operator $\mathcal A$ was introduced in Definition \ref{defA}.  
Note that due to the definition of $g_{\varepsilon}$, the function  $\tilde{q}_\varepsilon$  also depends on $q_1  ,q'_1$.
The statement below is an equivalent of Lemma~\ref{toy} for $\tilde{q}_{\varepsilon}$, comparing $\tilde{q}_{\varepsilon}$ to the ``target geodesic'' $\overline{q}$.
\begin{proposition} \label{apr1}
There exists  $\varepsilon_1>0$ such that, for any $\varepsilon\in(0,\varepsilon_1]$, for any $ ( {q}_1 , {q}'_1 ) $ in $ \overline{B}\big( (q_0 , q'_0), r_2  \big)$, 
the solution $\tilde{q}_\varepsilon$ given by \eqref{tildeq} lies in the ball  ${B} (q_0 , r_1  )$ at least up to $T$.
Moreover
 $\|\tilde{q}_{\varepsilon}-\bar{q}\|_{T, \varepsilon}   $  converges to $0$   and 
 $( \tilde{q}_{\varepsilon} , \tilde{q}'_{\varepsilon} )  (T)  $  converges to $  ( q_1  ,q'_1)$
 when $\varepsilon$   converges to $0^{+}$,  uniformly for  $ ( {q}_1 , {q}'_1 ) $ in $ \overline{B}\big( (q_0 , q'_0), r_2  \big)$ for both convergences.
\end{proposition} 
We recall that the norm $\| \cdot \|_{T, \varepsilon}$ was defined in \eqref{NormBonnasse}. 
The proof of Proposition \ref{apr1} can be found in Subsection \ref{SSpa1}.

The following result allows us to deduce the closeness of the trajectories $q_{\varepsilon,0}$, given by (\ref{renformEPS}) with $\gamma=0$, and $\tilde{q}_\varepsilon$ given by (\ref{tildeq}). Let us recall that by the definition of $T_{\varepsilon,\gamma}$ that comes along \eqref{renformEPS},  $q_{\varepsilon,0} $ lies in $B(q_0 , r_1) $ up to the time $T_{\varepsilon,0}$, which depends on $q_1  ,q'_1 $.

%
%
%
%
\begin{proposition} \label{apr2}
There exists $\varepsilon_2$ in $(0,\varepsilon_1]$ such that for any $\varepsilon \in (0,\varepsilon_2]$, one has $T_{\varepsilon,0} \geq T$.
Moreover $\|\tilde{q}_{\varepsilon}-q_{\varepsilon,0}\|_{C^1([0,T])} \rightarrow 0$ when ${\varepsilon\to0^{+}}$,  
 uniformly  for  $ ( {q}_1 , {q}'_1 ) $ in $ \overline{B}\big( (q_0 , q'_0), r_2  \big)$.
\end{proposition}
The proof of Proposition \ref{apr2} can be found in Subsection \ref{SSpa2}.

Finally, we have the following estimation of the deviation due to the circulation $\gamma$, which will be proved in Subsection \ref{SS34}.
\begin{proposition}\label{gimpact}
There exists $\varepsilon_3$ in $(0,\varepsilon_2]$  such that 
for all $\varepsilon \in (0,\varepsilon_{3}]$, there exists $\gamma_0 >0$  
 such that 
for any  $\gamma\in[-\gamma_0,\gamma_0]$,  we have $T_{\varepsilon,\gamma}\geq T$ and
$\|q_{\varepsilon,\gamma}-q_{\varepsilon,0}\|_{C^1 [0,T] }   $  converges to $0$ 
  when $ \gamma \to 0$,  
 uniformly for  $ ( {q}_1 , {q}'_1 ) $ in $ \overline{B}\big( (q_0 , q'_0), r_2  \big)$.
\end{proposition}

Propositions \ref{apr1}, \ref{apr2} and \ref{gimpact} give us directly the result of Proposition \ref{apmain}.
\subsection{Proof of Proposition \ref{apr1}}
\label{SSpa1}
We proceed as in the proof of Lemma \ref{toy} with a few extra complications related to the fact that the right hand side of the equation \eqref{tildeq} is more involved than the one of the equation \eqref{RER}  and  to the fact that we need to obtain uniform convergences with respect to $ ( {q}_1 , {q}'_1 ) $ in $ \overline{B}\big( (q_0 , q'_0), r_2  \big)$.
 
 As in the proof of Lemma \ref{toy} we introduce, for  $\eps$ in $(0,1)$, the time  $T_{\varepsilon} = \sup\,  \{  \hat{T} > 0   \text{ such that  } \,  \tilde q_{\varepsilon} (t) \in B(q_0 , r_1)   \text{ for } t \in (0,\hat{T}) \}$ 
and we first prove that there exists $\tilde{T} > 0$ such that for any $\eps$ in $(0,1)$,  $T_{\varepsilon} \geq \tilde{T}$ thanks to an energy estimate. 
In order to deal with the term coming from \eqref{vendredi} in the right hand side of the energy estimate,
recalling Remark~\ref{RemReg} and the definition of $F_{1,a}$ in \eqref{f1a},
we observe that for any $R > 0$, 
there exists $C > 0$ such that for any $q, \tilde q$ in $\mathcal{Q}_\delta$, for any $v$ in $B(0,R)$, 
\begin{equation}
| \, F_{1,a} (q) \big[\mathcal  A [q ,\overline g(\tilde q,v) ] \big] \, |   \leq C .
\end{equation}
This allows to deduce from the expressions of $v_0 $ and $v_1$  in \eqref{defv}
that there exists $\tilde{T} > 0$ and $C>0$ 
such that
 for any $ ( {q}_1 , {q}'_1 ) $  in $ \overline{B}\big( (q_0 , q'_0), r_2  \big)$,  for any $\eps$ in $(0,1)$,  $T_{\varepsilon} \geq \tilde{T}$ 
 and $ \|\tilde{q}'_{\varepsilon}\|_{C([0,{T}_{\varepsilon}])} \leq C$. 
We deduce that for  $ \eps_1 := \tilde{T} / 2 $ and $\eps$ in $(0,\eps_1)$,  $T_{\varepsilon} \geq  2 \varepsilon$ and 
 that   $\sup_{t \in [0, 2\varepsilon]} \,  | \tilde q_{\varepsilon} - q_0  | $ converges to $0$ as $\varepsilon$ goes to $0$  uniformly in  $( q_1  ,q'_1)$ in $ \overline{B}\big( (q_0 , q'_0), r_2  \big)$.

Now let us prove that 
  $ \tilde{q}_{\varepsilon}'(2\varepsilon) $  converges to $c_0 (q_1)$ as $\varepsilon$ goes  to $0$ 
   uniformly in  $( q_1  ,q'_1)$ in $ \overline{B}\big( (q_0 , q'_0), r_2  \big)$.
We integrate the equation \eqref{tildeq} on $[0, 2\varepsilon]$. 
  Thus 
\begin{multline} \label{tildeqInt}
 \Big(\mathcal{M}_g+ \mathcal{M}_a( \tilde{q}_{\varepsilon}  (2\varepsilon)  )\Big)  \tilde{q}'_{\varepsilon} (2\varepsilon) 
=  \Big(\mathcal{M}_g+\mathcal{M}_a( q_0)\Big) {q}'_0  \\
- \int_0^{2\varepsilon} \Big( D\mathcal{M}_a( \tilde{q}_{\varepsilon}) \cdot   \tilde{q}'_{\varepsilon}  \Big) \cdot \tilde{q}'_{\varepsilon} \, dt
- \int_0^{2\varepsilon}\langle\Gamma (\tilde{q}_{\varepsilon}),\tilde{q}'_{\varepsilon},\tilde{q}'_{\varepsilon}\rangle \, dt
+ \int_0^{2\varepsilon}  F_{1,a} (\tilde{q}_{\varepsilon}) \big[  \mathcal A[\tilde{q}_{\varepsilon},g_{\varepsilon}] \big]  \, dt .
\end{multline}
Then we pass  to the limit as  $\varepsilon$ goes to $0^+$  in the last equality.
Here we use two extra arguments  with respect to the corresponding argument in the proof of Lemma \ref{toy}.
On the one hand we see that the convergences of $ \mathcal{M}_a( \tilde{q}_{\varepsilon}  (2\varepsilon) )$ to 
$ \mathcal{M}_a( q_0 ) $
and of the two first terms of the last line to $0$, already obtained in the proof of Lemma \ref{toy}, 
hold uniformly with respect to  $( q_1  ,q'_1)$ in $ \overline{B}\big( (q_0 , q'_0), r_2  \big)$, as a consequence of the uniform estimates of $\tilde q_{\varepsilon}- q_0$ and  $ \tilde{q}_{\varepsilon}' $ obtained above.
On the other hand  the 
term $  F_{1,a} $ enjoys the following  regularity property with respect to $q$: 
   we have that 
   $ q \mapsto F_{1,a} (q) \Big[ \mathcal  A [q , \overline{g} (q_0, v) ]\Big] $ is  Lipschitz with respect to $q$ in $\mathcal{Q}_\delta$ uniformly 
   for $v$ in bounded sets of $\mathbb \R^3$. 
   Therefore  using that 
     $\sup_{t \in [0, 2\varepsilon]} \,  | \tilde q_{\varepsilon} - q_0  | $ converges to $0$ as $\varepsilon$ goes to $0$  uniformly in  $( q_1  ,q'_1)$ in $ \overline{B}\big( (q_0 , q'_0), r_2  \big)$, 
     the  expressions of $v_0 $ and $v_1$  in \eqref{defv} and that 
     $F_{1,a} (q_0) \Big[ \mathcal  A [q_0 ,\overline{g} (q_0, -2v_0)  \Big] = v_0$, according to Proposition~\ref{farcontr} we deduce that 
     $$\sup_{t \in [0, 2\varepsilon]} \, 
    \Big| F_{1,a} (\tilde{q}_{\varepsilon}) \Big[ \mathcal  A [\tilde{q}_{\varepsilon} , \overline{g} (q_0, -2v_0) ]\Big]  - v_0  \Big|$$
   converges to $0$ as $\varepsilon$ goes to $0$ uniformly in   $( q_1  ,q'_1)$ in $ \overline{B}\big( (q_0 , q'_0), r_2  \big)$.
   Since for $t$ in  $[0, 2\varepsilon]$, the equation   \eqref{vendredi} applied to $q=\tilde{q}_{\varepsilon}$  is simplified into
\begin{equation*}
F_{1,a} (\tilde{q}_{\varepsilon}) [\mathcal  A [\tilde{q}_{\varepsilon} ,g_{\varepsilon} ]] = 
\beta_{\varepsilon}^2 (t) F_{1,a} (\tilde{q}_{\varepsilon}) \Big[ \mathcal  A [\tilde{q}_{\varepsilon} ,\overline{g} (q_0, -2v_0)  ] \Big]  ,
\end{equation*}
and that 
  $\int_{0}^{2\varepsilon} \beta_{\varepsilon}^2 (t)\, dt=1$, 
   we get that the last term in \eqref{tildeqInt} converges to $ v_0$ when $\varepsilon$ goes to $0$.
   Moreover, due to the choice of $v_0$ the first and last term of the right hand side of  \eqref{tildeqInt}  can be combined at the limit to get 
   $ \Big(\mathcal{M}_g+\mathcal{M}_a( q_0)\Big) c_0 (q_1)$.
   
  Therefore, inverting the matrix in the right hand side of \eqref{tildeqInt} and passing to the limit, we see that
    $ \tilde{q}_{\varepsilon}'(2\varepsilon) $  converges to $c_0 (q_1)$ as $\varepsilon$ goes  to $0$ 
   uniformly in  $( q_1  ,q'_1)$ in $ \overline{B}\big( (q_0 , q'_0), r_2  \big)$.
  
  When  $t$ is in 
$[ 2\varepsilon,T-2\varepsilon]$, the equation 
\eqref{tildeq} reduces to a geodesic equation so that the same arguments as in  the proof of Lemma \ref{toy} apply. 

Finally for the last step, for $t$  in 
$[ T-2\varepsilon,T]$, we proceed   in the same way as in the first step.
This ends the proof of Proposition \ref{apr1}.
\subsection{Proof of Proposition \ref{apr2}}
\label{SSpa2}
We begin with the following lemma, which provides a uniform boundedness for the trajectories $q_{\varepsilon,0}$ satisfying (\ref{renformEPS}) with $\gamma=0$, that is
\begin{multline} \label{renff}
\big( \mathcal{M}_g + \mathcal{M}_a (q_{\varepsilon,0}) \big) q_{\varepsilon,0}''
+ \langle \Gamma (q_{\varepsilon,0}) , q_{\varepsilon,0}',q_{\varepsilon,0}'\rangle
=  F_{1,a} (q_{\varepsilon,0}) \big[   \mathcal A[q_{\varepsilon,0} ,g_{\varepsilon} ] \big]
\\ + F_{1,b} (q_{\varepsilon,0},q_{\varepsilon,0}') \big[   \mathcal A[q_{\varepsilon,0} ,g_{\varepsilon} ] \big]
+ F_2 (q_{\varepsilon,0}) \big[\partial_t  \mathcal A[q_{\varepsilon,0} ,g_{\varepsilon} ]  \big] .
\end{multline}
We recall that $g_\varepsilon$ is given by \eqref{gcontrol} with $v_0$ and $v_1$  given by \eqref{defv}. The terms $F_{1,a}$ and $F_{1,b}$ were defined in \eqref{f1a}-\eqref{f1b}, $F_{2}$ in \eqref{ef2}.
Also we recall that by definition of $T_{\varepsilon,0}$ (see the definition of $T_{\varepsilon,\gamma}$ in the end of Subsection~\ref{REM}), during the time interval $[0,T_{\varepsilon,0}]$, $q_{\varepsilon,0}$ remains in $B(q_0,r_1)$. 

%
\begin{lemma}\label{bddtrajg}
There exists $\varepsilon_a>0$ such that
$$\sup_{\substack{{(q_1,q'_1) \in \overline{B}((q_0,q'_0),r_2),} \\ {\varepsilon\in(0,\varepsilon_a]}}} \|q'_{\varepsilon,0}\|_{C([0,T_{\varepsilon,0}])}<+\infty.$$
\end{lemma}
\begin{proof}
First we see that
the mappings   
$$
q \mapsto F_{1,a} (q) [\mathcal A[q  , \overline{g}(q_0, v)  ]    ]  \text{ and } 
q \mapsto  F_{1,b} (q,\cdot) [ \mathcal A[q  , \overline{g}(q_0, v)  ]   ]
$$
are bounded  for  $q$ in $\mathcal{Q}_\delta$, uniformly for $v$ in bounded sets of $\mathbb \R^3$.
Let us now focus on the $F_2$ term. 
For $t$ in $[0,2 \varepsilon]$, 
$g_{\varepsilon} (t) = \beta_{\varepsilon} (t) \overline{g} (q_0, -2v_0)  $ 
so that, by  the chain rule, for $t$ in $[0,\min(2 \varepsilon,T_{\varepsilon,0}) ]$,
$$\partial_t  \mathcal A[q_{\varepsilon,0} ,g_{\varepsilon} ] = 
\beta_{\varepsilon}   D_q  \mathcal A[q_{\varepsilon,0} , \overline{g}(q_0, -2 v_0)   ]  \cdot q'_{\varepsilon,0} 
+ \beta'_{\varepsilon}  \mathcal A[q_{\varepsilon,0} , \overline{g}(q_0, -2 v_0) ].$$
For what concerns $F_2$ we have, using the property \eqref{csR},
\begin{multline*} 
F_2  (q_{\varepsilon,0}) \big[\partial_t  \mathcal A[q_{\varepsilon,0} ,g_{\varepsilon} ]  \big]
=  \beta_{\varepsilon}  \int_{ \partial \mathcal{S} (q_{\varepsilon,0})} \Big(D_q  \mathcal A[q_{\varepsilon,0} , \overline{g}(q_0, -2 v_0)   ]  \cdot q'_{\varepsilon,0}\Big) \,  \partial_n \Phi(q_{\varepsilon,0},\cdot) \, d \sigma \\ 
+ \beta'_{\varepsilon}  \Big( \int_{ \partial \mathcal{S} (q_{\varepsilon,0})} \mathcal A[q_{\varepsilon,0} , \overline{g}(q_0, -2 v_0) ]
  \,  \partial_n \Phi(q_{\varepsilon,0},\cdot) \, d \sigma 
-  \int_{ \partial \mathcal{S} (q_{0})} \mathcal A[q_{0} ,\overline{g}(q_0, -2 v_0) ]
\,  \partial_n \Phi(q_{0},\cdot) \, d \sigma \Big) .
\end{multline*}
%

Using that the mapping 
$q \mapsto  \int_{ \partial \mathcal{S} (q)} \nabla_q \mathcal A[q  , \overline{g}(q_0, v)  ]  \otimes    \partial_n \Phi(q,\cdot) \, d \sigma$
 is  bounded  for $q$ over $\mathcal{Q}_\delta$ 
and that  the mapping 
 $q \mapsto  \int_{ \partial \mathcal{S} (q)}  \mathcal A[q  , \overline{g}(q_0, v)  ]   \,  \partial_n \Phi(q,\cdot) \, d \sigma$
 is  Lipschitz with respect to $q$ in $\mathcal{Q}_\delta$, both uniformly 
   for $v$ in bounded sets of $\mathbb \R^3$, 
 we see that this involves (recalling the expression of $\beta_{\varepsilon} $ given at the beginning of Section \ref{Illu})
\begin{equation} \label{EstF2}
\big| F_2  (q_{\varepsilon,0}) \big[\partial_t  \mathcal A[q_{\varepsilon,0} ,g_{\varepsilon} ]  \big] \big|
\lesssim
C \left(  \frac{1}{\varepsilon^{1/2}} |q'_{\varepsilon,0}| +
\frac{1}{\varepsilon^{3/2}} |q_{\varepsilon,0} - q_{0}|  \right),
\end{equation}
 uniformly  for  $ ( {q}_1 , {q}'_1 ) $ in $ \overline{B}\big( (q_0 , q'_0), r_2  \big)$.
Then, multiplying  \eqref{renff}  by $q'_{\varepsilon,0}$ and 
using once more the identity  \eqref{identityNRJ}, we obtain, for any $\eps$ in $(0,1)$, for $t$ in $[0,\min(2 \varepsilon,T_{\varepsilon,0}) ]$,
\begin{multline} \label{EEs}
 \Big(\mathcal{M}_g+ \mathcal{M}_a(q_{\varepsilon,0} (t))\Big) q_{\varepsilon,0}' (t) \cdot q_{\varepsilon,0}'  (t)  = 
 \Big(\mathcal{M}_g+\mathcal{M}_a( q_0)\Big) {q}'_0 \cdot  {q}'_0 
 \\  +  2 \int_0^{t} \Big(     F_{1,a} (q_{\varepsilon,0}) \big[   \mathcal A[q_{\varepsilon,0} ,g_{\varepsilon} ] \big]  +  F_{1,b} (q_{\varepsilon,0},q_{\varepsilon,0}') \big[   \mathcal A[q_{\varepsilon,0} ,g_{\varepsilon} ] \big]  + F_2 (q_{\varepsilon,0}) \big[\partial_t  \mathcal A[q_{\varepsilon,0} ,g_{\varepsilon} ]  \big]\Big) \cdot  q' _{\varepsilon,0}     ,
\end{multline}
Then, using \eqref{identityNRJeps}, the boundedness of 
the mappings    $ q \mapsto F_{1,a} (q) [\mathcal A[q  , \overline{g}(q_0, v)  ]    ]  $ and 
 $q \mapsto  F_{1,b} (q,\cdot) [ \mathcal A[q  , \overline{g}(q_0, v)  ]   ] $ already mentioned above, the definition of $\beta_{\varepsilon}$ and the bound  \eqref{EstF2}, we get 
\begin{equation*}
  |q_{\varepsilon,0}' (t)  |^2 \leq C \left(
1  +  \frac{1}{\varepsilon^{1/2}} \int_0^{t}   |q_{\varepsilon,0}' (s)  |^2 \, ds  
 +  \frac{1}{\varepsilon^{3/2}}  \int_0^{t}  |q'_{\varepsilon,0} (s)|  |q_{\varepsilon,0} (s)- q_{0}|  \, ds  \right).
\end{equation*}
Then using the mean value theorem and that $t \leq 2 \varepsilon$, we have that 
\begin{equation*}
  |q_{\varepsilon,0}' (t)  |^2 \leq C  
\left(1  +  \varepsilon^{1/2}  \sup_{[0,\min(2 \varepsilon,T_{\varepsilon,0}) ]} \,    |q_{\varepsilon,0}'  |^2  \right) ,
\end{equation*}
so that for $\varepsilon$ small enough, and for $t$ in $[0,\min(2 \varepsilon,T_{\varepsilon,0}) ]$,
$  |q_{\varepsilon,0}' (t)  | \leq C  $,  uniformly  for  $ ( {q}_1 , {q}'_1 ) $ in $ \overline{B}\big( (q_0 , q'_0), r_2  \big)$.
As a consequence of the usual blow-up criterion for ODEs, we have that  $T_{\varepsilon,0} \geq 2 \varepsilon$.

During the next phase, i.e. for $t$ in  $[2 \varepsilon,T - 2 \varepsilon]$, 
the control is inactive so that  the equation  \eqref{renff}  is a geodesic equation. 
Then by a simple energy estimate we get again that $  |q_{\varepsilon,0}' (t)  | \leq C  $ on $[0,\min(T-2 \varepsilon,T_{\varepsilon,0}) ]$.

Finally if $T_{\varepsilon,0} \geq T-2 \varepsilon$, then 
we deal with  the last phase as in the first phase. 
This concludes the proof of Lemma~\ref{bddtrajg}. 
\end{proof}

We then conclude the proof of Proposition \ref{apr2} by a classical comparison argument using Gronwall's lemma and the Lipschitz regularity with respect to $q$ 
of the various mappings involved ($ \mathcal{M}_a$, $\Gamma$, $ F_{1,a} $, $ F_{1,b} $ and $F_2$).
This allows to prove that there exists $\varepsilon_2$ in $(0,\varepsilon_1]$ such that for any $\varepsilon \in (0,\varepsilon_2]$,   
$T_{\varepsilon,0} \geq T$ and $\|\tilde{q}_{\varepsilon}-q_{\varepsilon,0}\|_{C^1([0,T])} \rightarrow 0$ when ${\varepsilon\to0^{+}}$,  
 uniformly  for  $ ( {q}_1 , {q}'_1 ) $ in $ \overline{B}\big( (q_0 , q'_0), r_2  \big)$.
This ends the proof of Proposition \ref{apr2}. 
\subsection{Proof of Proposition \ref{gimpact}}\label{SS34}

First we may extend Lemma \ref{bddtrajg} to the solutions $q_{\varepsilon,\gamma}$ to (\ref{renformEPS}) in the following manner.

\begin{lemma}\label{bddtrajg++}
There exists $\varepsilon_b$ in $(0,\varepsilon_2)$ such that 
 $\|q'_{\varepsilon,\gamma}\|_{C([0,T_{\varepsilon,\gamma}])}$ is bounded uniformly in $\varepsilon\in(0,\varepsilon_b]$,
for any  $\gamma\in[-1,1]$, and 
for $(q_1,q'_1) \in \overline{B}((q_0,q'_0),r_2)$.
\end{lemma}
It is indeed a matter of adding the ``electric field'' $E$ in \eqref{EEs}, and noting that $E$ is bounded on $Q_{\delta}$; the ``magnetic field'' $B$ gives no contribution to the energy. \par
\ \par
We now finish the proof of Proposition \ref{gimpact}.
Using a comparison argument we obtain that
there exists $\varepsilon_3$ in $(0,\varepsilon_b]$  such that 
for all $\varepsilon \in (0,\varepsilon_{3}]$, 
 there exists $\gamma_0 >0$  
 such that 
for any  $\gamma\in[-\gamma_0,\gamma_0]$,  we have $T_{\varepsilon,\gamma}\geq T$ and
$\|q_{\varepsilon,\gamma}-q_{\varepsilon,0}\|_{C^1 [0,T] }   $  converges to $0$ 
  when $ \gamma \to 0$,  
 uniformly for  $ ( {q}_1 , {q}'_1 ) $ in $ \overline{B}\big( (q_0 , q'_0), r_2  \big)$.
This concludes the proof of Proposition~\ref{gimpact}.
%
%
%
%
%
%
%
%
\section{Design of the control according to the solid position. Proof of Proposition \ref{farcontr}}
\label{S4}

This section is devoted to the proof of Proposition \ref{farcontr}.
\subsection{The case of a homogeneous disk}
Before proving Proposition~\ref{farcontr} we establish the following similar result concerning the simpler case where the solid is a homogeneous disk.
In that case, the statement merely considers $q$ of the form $q=(h,0)$.
Thus in order to simplify the writing, we introduce
\begin{equation*}
{\mathcal Q}_{\delta}^h :=\{ h \in \R^2 \    \text{ such that }  \ (h,0) \in {\mathcal Q}_{\delta} \}.
\end{equation*}
Also in all this section when we will write $q$, it will be understood that $q$ is associated with $h$ by $q=(h,0)$.
\begin{proposition} \label{dfarcontr}
Let $\delta>0$. Then there exists a continuous mapping $\overline{g}: {\mathcal Q}_{\delta}^h  \times \mathbb R^2  \rightarrow {\mathcal C}$ such that the function $\overline{\alpha} := \mathcal A [q,\overline{g} (q,v)]$ in $C^\infty (\overline{\mathcal{F}(q)};\mathbb{R})$  satisfies:
\begin{gather}
\label{harm}
\Delta \overline{\alpha}(q,x)=0   \text{ in } \mathcal{F}(q),  \text{ and } \partial_{n}\overline{\alpha}(q,x)=0
\text{ on } \partial\mathcal{F}(q)\setminus\Sigma, \\
\label{hsz}
\int_{ \partial \mathcal{S} (q)}  |\nabla\overline{\alpha}(q,x)|^{2} \, n \, d \sigma  = v, \\
\label{cs}
\int_{\partial\mathcal{S}(q)} \overline{\alpha}(q,x) \, n \, d\sigma = 0 .
\end{gather}
\end{proposition}
In order to prove Proposition~\ref{dfarcontr}, the mapping $\overline{g}$ will be constructed using a combination of some elementary functions which we introduce in several lemmas.

To begin with, we will make use of the elementary geometrical property that $ \{n (q_0,x):\ x\in\partial\mathcal{S}(q_0)\}$ is the unit circle $\mathbb S^1$ and of the following lemma.
\begin{lemma} \label{kup1}
There exist three vectors $e_1, e_{2}, e_{3} \in \{n (q_0,x):\ x\in\partial\mathcal{S}(q_0)\}$
and positive $C^\infty$ maps $(\mu_i)_{1\leqslant i \leqslant 3} : \R^{2} \rightarrow \R_{+}$
such that for any $v \in \R^{2}$,
\begin{equation} \label{SommeDesMu}
\sum_{i=1}^3 \mu_i (v) e_i = v.
\end{equation}
\end{lemma}
\begin{proof}
One may consider for instance $e_{1}:=(1,0)$, $e_{2}:=(0,1)$, $e_{3}:=(-1,-1)$, and
\begin{multline*}
\mu_{1}(v) = v_{1} +\sqrt{1 + |v_{1}|^2 + |v_{2}|^2}, \ \
\mu_{2}(v) = v_{2} +\sqrt{1 + |v_{1}|^2 + |v_{2}|^2} \\
\text{ and } \mu_{3}(v) =  \sqrt{1 + |v_{1}|^2 + |v_{2}|^2}.
\end{multline*}
\end{proof}

In the next lemma, we introduce some functions that are defined in a neighbourhood of $\partial\mathcal{S}(q_0)$ (for some $q_{0} = (h_{0},0)$ fixed), satisfying some counterparts of the properties \eqref{harm} and \eqref{hsz}.
\begin{lemma} \label{LemBase}
There exist families of functions $(\tilde{\alpha}_{\varepsilon}^{i,j})_{\varepsilon \in (0,1)}$, $i,j\in\{1,2,3\}$, such that 
 for any  $i,j\in\{1,2,3\}$, for any $\varepsilon \in (0,1)$, $\tilde{\alpha}_{\varepsilon}^{i,j}$ is defined and 
harmonic in a closed neighbourhood $ \mathcal{V}_{\varepsilon}^{i,j}$ of $\partial\mathcal{S}(q_0)$,
satisfies $\partial_{n} \tilde{\alpha}_{\varepsilon}^{i,j} =0$ on $\partial\mathcal{S}(q_0)$,
and moreover one has for any  $i,j,k,l$ in $\{1,2,3\}$, 
$$
\int_{\partial\mathcal{S}(q_0)} \nabla\tilde{\alpha}_{\varepsilon}^{i,j} \cdot \nabla\tilde{\alpha}_{\varepsilon}^{k,l} \,n\, d\sigma
\rightarrow  \delta_{(i,j),(k,l)} \, e_i   \quad \text{ as } \eps  \rightarrow 0^+ .
$$
\end{lemma}
\begin{proof}
Without loss of generality, we may suppose that $\mathcal{S}(q_0)$ is the unit disk.
Consider the parameterisation $\{c(s)=(\cos(s),\sin(s)),\ s\in [0,2\pi]\}$ of $\partial\mathcal{S}(q_0)$
and the corresponding $s_i$ such that $n(q_0,c(s_i))=e_i,\ i\in\{1,2,3\}$. \par
We consider families of smooth functions $\beta_{\varepsilon}^{i,j}: [0,2\pi] \to \mathbb{R}$, $i,j\in\{1,2,3\}$, $\varepsilon \in (0,1)$, such that
$\text{supp } \beta_{\varepsilon}^{i,j} \, \cap \, \text{supp }\beta_{\varepsilon}^{k,l}=\emptyset$ whenever $(i,j)\neq(k,l)$, 
$\text{diam}\left(\text{supp }\beta_{\varepsilon}^{i,j}\right)\to 0$ as $\varepsilon\to0^{+}$,
\begin{equation*}
  \int_{0}^{2\pi}\beta_{\varepsilon}^{i,j}(s)\, d\sigma= 0   \, \text{  and } 
\left| \int_{0}^{2\pi} |\beta_{\varepsilon}^{i,j}(s)|^2 \, n(q_0,c(s)) ds-e_i \right| \to 0  \text{ as } \varepsilon \to 0^{+}.
\end{equation*}
Then we define $\tilde{\alpha}_{\varepsilon}^{i,j}$ in polar coordinates as the truncated Laurent series:
$$
\tilde{\alpha}_{\varepsilon}^{i,j}(r,\theta) := \frac12 \sum_{0<k\leq K} \frac{1}{k} \left(r^k+\frac{1}{r^k}\right)( -\hat{b}_{k,\varepsilon}^{i,j} \cos(k\theta) +\hat{a}_{k,\varepsilon}^{i,j}  \sin(k\theta)) , $$
where $\hat{a}_{k,\varepsilon}^{i,j}$ and $\hat{b}_{k,\varepsilon}^{i,j}$ denote the $k$-th Fourier coefficients of the function $\beta_{\varepsilon}^{i,j} $.
It is elementary to check that the function $\tilde{\alpha}_{\varepsilon}^{i,j}$ satisfies the required properties for an appropriate choice of $K$.
\end{proof}
Now, for any $h \in\mathcal{Q}^h_\delta$, we may define  $\mathcal{V}_{\varepsilon}^{i,j}(q):= \mathcal{V}_{\varepsilon}^{i,j}-h_0+h,$
which is a neighborhood of $\partial\mathcal{S}(q)$, 
and  $\tilde{\alpha}_{\varepsilon}^{i,j}(q,x):=\tilde{\alpha}_{\varepsilon}^{i,j}(x+h_0-h), \text{ for each } x\in \mathcal{V}_{\varepsilon}^{i,j}(q)$.
We have for $i,j,k,l$ in $\{1,2,3\}$, 
$$
\int_{\partial\mathcal{S}(q)}  \nabla\tilde{\alpha}_{\varepsilon}^{i,j}  (q,x) \cdot \nabla\tilde{\alpha}_{\varepsilon}^{k,l} (q,x)  \, n(q,x) \, d\sigma
=
\int_{\partial\mathcal{S}(q_0)} \nabla\tilde{\alpha}_{\varepsilon}^{i,j} (x)  \cdot \nabla\tilde{\alpha}_{\varepsilon}^{k,l} (x)  \, n(q_0,x) \, d\sigma .
$$
Proceeding as in \cite{OG-Addendum} (see also \cite[p. 147-149]{Cortona}) and relying in particular Runge's theorem, we have the following result which asserts the existence of harmonic approximate extensions on the whole fluid domain.
\begin{lemma} \label{LemmeEta}
There exists a family of functions $(\alpha_{\eta}^{i,j})_{\eta \in (0,1)}$, $i,j \in \{1,2,3\}$,  harmonic in $\mathcal{F}(q)$, 
satisfying  $\partial_n \alpha_{\eta}^{i,j}(q,x)=0$ on  $\partial\mathcal{F}(q)\setminus \Sigma$, with for any $k$ in $\mathbb N$, 
\begin{equation} \label{del2new}
\| \alpha_{\eta}^{i,j}(q,\cdot) - \tilde{\alpha}_{\varepsilon}^{i,j}(q,\cdot) \|_{C^k(\mathcal{V}_{\varepsilon}^{i,j}(q)\cap\overline{\mathcal{F}(q)})} 
\rightarrow   0     \text{ when } \eta \rightarrow  0^+ .
\end{equation}
\end{lemma}

We now check that the above construction can be made continuous in $q$.

 \begin{lemma} \label{Lem10}
For any $\nu>0$, there exist continuous mappings
$h \in \mathcal{Q}^h_\delta \mapsto \overline{\alpha}^{i,j} (q,\cdot)\in C^{\infty}(\overline{\mathcal{F}(q)})$ where $q=(h,0)$,
$i, j\in\{1,2,3\}$, such that for any $h \in \mathcal{Q}^h_\delta$,
$\Delta_{x}\overline{\alpha}^{i,j} (q,x)=0$ in $\mathcal{F}(q)$, $\partial_{n}\overline{\alpha}^{i,j} (q,x)=0$ on $\partial\mathcal{F}(q)\setminus\Sigma$ and
\begin{equation} \label{nd1}
\left| \int_{\partial\mathcal{S}(q)} \nabla \overline{\alpha}^{i,j}(q,\cdot) \cdot \nabla \overline{\alpha}^{k,l}(q,\cdot)  \,n\, d\sigma
- \delta_{(i,j),(k,l)} \, e_i  \right| \leq \nu .
\end{equation}
\end{lemma}

\begin{proof}
Let us assume that the functions $\alpha_{\eta}^{i,j}$ were previously defined not only for $h \in Q^h_{\delta}$ but for  $h \in \overline{Q^h_{\delta}}$; this is possible by using a smaller  $\delta$. 
Hence we may for each $h \in \overline{Q^h_{\delta}}$ find functions $\alpha_{\eta}^{i,j}$ (for some $\eta>0$) satisfying the properties above, and in particular such that \eqref{nd1} is valid. \par
Next we observe that for any $h \in \overline{\mathcal{Q}^h_\delta}$, setting $q=(h,0)$,
the unique solution $\hat{\alpha}_{\eta}^{i,j}(\tilde{q},q,\cdot)$ (up to an additive constant) to the Neumann problem
$\Delta_{x} \hat{\alpha}_{\eta}^{i,j}(\tilde{q},q,x)=0$ in $ \mathcal{F}(\tilde{q})$, $\partial_{n}\hat{\alpha}_{\eta}^{i,j}(\tilde{q},q,x)=0$ on $\partial \mathcal{F}(\tilde{q}) \setminus \Sigma$,
$\partial_{n} \hat{\alpha}_{\eta}^{i,j}(\tilde{q},q,x) = \partial_{n}\alpha_{\eta}^{i,j}(q,x)$ on $\Sigma$,  is continuous with respect to $\tilde{q} \in \mathcal{Q}_\delta$.
It follows that when a family of functions ${\alpha}_{\eta}^{i,j}$ satisfies \eqref{nd1} at some point $h \in \overline{Q^h_{\delta}}$, it satisfies \eqref{nd1} (with perhaps $2\nu$ in the right hand side) in some neighborhood of $h$. Since $\overline{Q^h_{\delta}}$ is compact and can be covered with such neighborhoods, one can extract a finite subcover and use a partition of unity (according to the variable $q$) adapted to this subcover to conclude: one gets an estimate like \eqref{nd1} with $C \nu$ on the right hand side (for some constant $C$). It is then just a matter of considering $\nu/C$ rather than $\nu$ at the beginning.
\end{proof}
Finally our basic bricks to prove Proposition~\ref{dfarcontr} are given in the following lemma, where we can add the constraint \eqref{cs}.
\begin{lemma} \label{Lem3emeCond}
For any $\nu>0$, there exist continuous mappings
$q=(h,0) \in\mathcal{Q}_\delta \mapsto \overline{\alpha}^{i} (q,\cdot)\in C^{\infty}(\overline{\mathcal{F}(q)})$, 
$i\in\{1,2,3\}$, such that for any $q=(h,0) \in\mathcal{Q}_\delta$,
$\Delta_{x} \overline{\alpha}^{i} (q,x)=0$ in $\mathcal{F}(q)$, $\partial_{n}\overline{\alpha}^{i} (q,x)=0$ on $\partial\mathcal{F}(q)\setminus\Sigma$ and
\begin{gather} \label{angou}
\left|  \int_{\partial\mathcal{S}(q)} \nabla \overline{\alpha}^{i}(q,\cdot) \cdot \nabla \overline{\alpha}^{j}(q,\cdot) \,n\, d\sigma
- \delta_{i,j} \, e_i  \right| \leq \nu ,  \\
\int_{\partial\mathcal{S}(q)}   \overline{\alpha}^{i}(q,\cdot) \,n\, d\sigma  = 0 .
\end{gather}
\end{lemma}

\begin{proof}
Consider the functions $\overline{\alpha}^{i,j}$ given by Lemma~\ref{Lem10}.
For any $q=(h,0) \in\mathcal{Q}_\delta$, for any  $i\in\{1,2,3\}$, the three vectors 
$ \int_{\partial\mathcal{S}(q)} \overline{\alpha}^{i,j}(q,\cdot) \, n \, d\sigma$, where $j\in\{1,2,3\}$, are 
linearly dependent in  $\mathbb R^2$; therefore there exists $\lambda^{i,j}(q)\in\mathbb{R}$ such that
\begin{equation} \label{fucklasncf}
\sum_{j=1}^{3} \lambda^{i,j}(q) \int_{\partial\mathcal{S}(q)} \overline{\alpha}^{i,j}(q,\cdot) \, n \, d\sigma = 0
\text{ and } \displaystyle\sum_{j=1}^{3} |\lambda^{i,j}(q)|^2 = 1,
\end{equation}
Then one defines $\overline{\alpha}^{i}(q,\cdot):=\sum_{j=1}^{3} \lambda^{i,j}(q)  \overline{\alpha}^{i,j}(q,\cdot)$, and one checks that it satisfies \eqref{angou} with some $C \nu$ in the right hand side. Again changing $\nu$ in $\nu/C$ allows to conclude.
\end{proof}
We are now in position to prove Proposition~\ref{dfarcontr}.
\begin{proof}[Proof of Proposition~\ref{dfarcontr}]
Let  $\delta>0$.
Let $\nu>0$. We define the mapping ${\mathcal S}$ which with $(h,v) \in \mathcal{Q}^h_\delta \times \R^2$ associates the function 
$$
\tilde{\alpha}(q,\cdot):=\sum_{i=1}^{3}\sqrt{\mu^{i}(v)} \, \overline{\alpha}^{i}(q,\cdot),
$$
in $C^\infty (\overline{\mathcal{F}(q)})$, where the functions $\mu^i$ were introduced in Lemma~\ref{kup1} and the functions $\overline{\alpha}^{i}$ were introduced in Lemma~\ref{Lem3emeCond}.
Next we define ${\mathcal T}: \mathcal{Q}_\delta^h \times \R^2 \rightarrow \mathcal{Q}_\delta^h \times \R^2$ by
\begin{equation*}
(h,v) \mapsto ({\mathcal T}_{1},{\mathcal T}_{2})(h,v):= \left( h, \int_{\partial\mathcal{S}(q)} |\nabla \tilde{\alpha}(q,\cdot)|^2 \,n\, d\sigma \right) , 
\ \text{ where } \ \tilde{\alpha} ={\mathcal S}(h,v) .
\end{equation*}
Using \eqref{SommeDesMu} and \eqref{angou}, one checks that ${\mathcal T}$ is smooth and that
\begin{equation*}
\frac{\partial {\mathcal T}_{2}}{\partial v} = \mbox{Id} + {\mathcal O}(\nu).
\end{equation*}
Hence taking $\nu$ sufficiently small, we see that $\frac{\partial {\mathcal T}_{2}}{\partial v}$ is invertible, hence $\frac{\partial {\mathcal T}}{\partial (h,v)}$ is invertible.
Consequently one can use the inverse function theorem on ${\mathcal T}$: for each $h_{0} \in \overline{{Q}_{\delta}^h}$ it realizes a local diffeomorphism at$(h_{0}, 0)$, and hence on $\overline{Q^h_{\delta}} \times B(0,r)$ for $r>0$ small enough. 
This gives the result of Proposition \ref{dfarcontr} for $v$ small: given $(h,v) \in \overline{Q^h_{\delta}} \times B(0,r)$, we let $(h,\tilde{v}):= {\mathcal T}^{-1}(h,v)$. Then the functions 
$\overline{\alpha}:=\sum_{i=1}^{3}\sqrt{\mu^{i}( \tilde{v})} \, \overline{\alpha}^{i}(q,\cdot)$ and $\overline{g}:=\mathbbm{1}_{\Sigma} \,  \partial_{n} \overline{\alpha}$ satisfy the requirements.
The general case follows by linearity of \eqref{harm} and \eqref{cs} and by homogeneity of \eqref{hsz}.
This ends the proof of Proposition \ref{dfarcontr}.
\end{proof}
\subsection{The case when $\mathcal{S}_0$ is not a disk}
We now get back to the proof of Proposition \ref{farcontr}.
We will denote by $\text{coni}(A)$ the conical hull of $A$, namely
$$
\text{coni}(A) : =\left\{\sum_{i=1}^{k} \lambda_i a_i,\ k \in \mathbb{N}^*, \ \lambda_i \geq 0, \ a_i\in A \right\},
$$
The first step is the following elementary geometric lemma.
\begin{lemma} \label{geocond}
Let $\mathcal{S}_0\subset\Omega$ bounded, closed, simply connected with smooth boundary, which is not a disk.
Then $ \text{coni}\{(n (x),(x-h_0)^\perp \cdot n(x)), \ x \in \partial \mathcal{S}_0 \}=\mathbb{R}^3.$
\end{lemma}
\begin{proof}
Suppose the contrary. Then there exists a plane separating (in the large sense) the origin in $\mathbb{R}^3$ from the set
$\text{coni} (\{ (n (x), (x-h_0)^\perp \cdot n(x)), \ x \in \partial \mathcal{S}_0 \})$.
We claim that a normal vector to this plane can be put in the form $(a,b,1)$, with $a,b\in\mathbb{R}$. Indeed, otherwise it would need to be of the form $(a,b,0)$, and the separation inequality would give $(a,b)\cdot n (x) \geq 0,\ \forall x \in \partial\mathcal{S}_0$.
However, since $\partial\mathcal{S}_0$ is a smooth, closed curve, the set $\{n (x):\ x\in\partial\mathcal{S}_0\}$ is the unit circle of $\mathbb{R}^2$, therefore we have a contradiction. \par
Now we deduce that we have the following separation property:
$$
(a,b)\cdot n (x) + (x-h_0)^\perp \cdot n(x) \geq 0, \ \ \forall  x \in \partial \mathcal{S}_0.
$$
Denoting $w=(a,b)-h_0^\perp$, this translates into $(w+x^\perp)\cdot n (x) \geq 0$. But using Green's formula, we get
\begin{equation*}
0 \leq \int_{\partial\mathcal{S}_0} (w+x^\perp)\cdot n (x) \, d\sigma = \int_{\mathcal{S}_0} \text{div}(w+x^\perp) \, dx = 0,
\end{equation*}
and consequently, we deduce that $(w+x^\perp)\cdot n (x)=0$ for all $x$ in $\partial \mathcal{S}_0$.
This is equivalent to  $(x-w^\perp) \cdot \tau (x) = 0$ for all $x$ in $\partial \mathcal{S}_0$.
Parameterizing the translated curve $\partial\mathcal{S}_0-w^\perp$ by $\{c(s),\ s\in[0,1]\}$, it follows that
$c(s)\cdot \dot{c}(s) = 0$, for all $s$ in $[0,1]$, and therefore $|c(s)|^2$ is constant.
This means that $\partial\mathcal{S}_0-w^\perp$ is a circle, so $\mathcal{S}_0$ is a disk, which is a contradiction.
\end{proof}
Fix $q_0 \in Q_{\delta}$. 
Recalling the definitions of the Kirchhoff potentials in   (\ref{phi})  and  (\ref{kir}), we infer from the previous lemma that 
$$  \text{coni}\{ \partial_n \Phi (q_0,x),\ x\in\partial\mathcal{S}_0 \} =   \mathbb{R}^3.$$
In place of Lemma~\ref{kup1}, we have the following lemma which is a straightforward consequence of  Lemma \ref{geocond} and of a repeated application of  Carath\'eodory's theorem on the convex hull.
\begin{lemma} \label{kup2}
There  are  some  $(x_i )_{i\in\{1,\ldots,16\}}$ in $ \partial \mathcal{S}_0$ 
and positive continuous mappings $\mu_{i}: \R^{3} \rightarrow \R, \ {1\leqslant i \leqslant 16} $,  $v \mapsto \mu_i(v) $ such that
$\sum_{i=1}^{16} \mu_i(v)  \partial_n \Phi (q_0,x_i)= v $.
\end{lemma}
We are now in position to establish Proposition~\ref{farcontr}.
We deduce from Lemma  \ref{kup2} that for any  $q := (h,\vartheta) \in \overline{Q_{\delta}}$, for any $v$ in $ \R^{3}$, 
$$\sum_{i=1}^{16} \mu_i( \mathcal{R} (\vartheta) v) \,   \partial_n \Phi (q, x_i (q)) =  \mathcal{R} (\vartheta)  v   ,$$
where $ x_i (q) := R(\vartheta) (x_i - h_0) + h $ 
and 
$\mathcal{R} (\vartheta)  $ denotes the $3 \times 3$ rotation matrix defined by
$$
\mathcal{R} (\vartheta) 
:= \left( \begin{array}{ccc}
R(\vartheta)  & 0\\
0 & 1 
\end{array} \right) . 
$$

Due to the Riemann mapping theorem, there exists a biholomorphic mapping 
 $\Psi:\overline{\mathbb{C}}\setminus B(0,1)\to \overline{\mathbb{C}}\setminus\mathcal{S}(q)$ with $\partial\mathcal{S}(q)=\Psi(\partial B(0,1))$, where $\overline{\mathbb{C}}$ denotes the Riemann sphere.
We consider the parametrisations $\{c(s)=(\cos(s),\sin(s)),\ s\in [0,2\pi]\}$ of $\partial B(0,1)$, respectively $\{\Psi(c(s)),\ s\in [0,2\pi]\}$ of $\partial \mathcal{S}(q)$, and the corresponding $s_i$ such that $ x_i (q) = \Psi(c(s_i))$, for $i\in\{1,\ldots,16\}$. 

Then, for any smooth function $\alpha:\partial\mathcal{S}(q)\to\mathbb{R}$, due to the Cauchy-Riemann relations, we have the following:
\begin{align*}
\begin{split}
\partial_{n} \alpha(\Psi(x)) &= \frac{1}{\sqrt{|\text{det}(D\Psi(x))|}} \partial_{n_B} (\alpha\circ\Psi )(x),\\
\int_{\partial\mathcal{S}(q)} |\nabla\alpha(x)|^2 \, \partial_n \Phi(q,x) \, d\sigma &= \int_{\partial B(0,1)} |\nabla\alpha(\Psi(x))|^2 \, \partial_{n_B} \Phi(q,\Psi(x)) \, \frac{1}{\sqrt{|\text{det}(D\Psi(x))|}} \, d\sigma,
\end{split}
\end{align*}
for any $x\in\partial B(0,1)$, where $n$ and $n_B$ respectively denote the normal vectors on $\partial\mathcal{S}(q)$ and $\partial B(0,1)$. Note that, since $\Psi$ is invertible, we have $|\text{det}(D\Psi(x))|>0$, for any $x\in\partial B(0,1)$.

For each $\varepsilon>0$, $i\in\{1,\ldots,16\}, j\in\{1,2,3,4\}$ (here the index $j$ belongs to $\{1,2,3,4\}$ rather  than  $\{1,2,3\}$ in order to adapt the  linear dependence argument of Lemma~\ref{Lem3emeCond} to the case of the three linear constraints \eqref{csR}), we consider families of smooth functions $\beta_{\varepsilon}^{i,j}:[0,2\pi]\to\mathbb{R}$ satisfying
$\text{supp }\beta_{\varepsilon}^{i,j}\cap\text{supp }\beta_{\varepsilon}^{k,l}=\emptyset $  for $(i,j)\neq(k,l),$
$\text{diam}\left(\text{supp }\beta_{\varepsilon}^{i,j}\right)\to 0$ as $\varepsilon\to0^{+}$,
\begin{equation*}
\int_{0}^{2\pi}\beta_{\varepsilon}^{i,j}(s)  \,  ds=0,
\end{equation*}
and
\begin{equation*}
\left|\int_{0}^{2\pi}|\beta_{\varepsilon}^{i,j}(s)|^2 \, \partial_n \Phi(q,c(s)) \frac{1}{\sqrt{|\text{det}(D\Psi(c(s)))|}}  \,\, ds-  \tilde{e}_{i}   \right|\to 0\ \text{as } \varepsilon\to0^{+} ,
\end{equation*}
where 
$$\tilde{e}_{i}:=\frac{1}{\sqrt{|\text{det}(D\Psi(c(s_i)))|}} \partial_n \Phi (q, x_i (q)).$$

Then one may proceed  essentially  as in the proof of Proposition \ref{dfarcontr}.
The details are therefore left to the reader.
%
%
%
%
%
%
%
%
%

\section*{Acknowledgements}

We would like to thank Jimmy Lamboley and Alexandre Munnier for helpful conversations on shape differentiation.
The authors also thank  the Agence Nationale de la Recherche, Project DYFICOLTI, grant ANR-13-BS01-0003-01 and Project IFSMACS, grant ANR-15-CE40-0010  for their financial support. 
F. Sueur  
was also partially supported by the Agence Nationale de la Recherche, Project SINGFLOWS grant ANR-18-CE40-0027-01, Project BORDS, grant ANR-16-CE40-0027-01, the Conseil R\'egionale d'Aquitaine, grant 2015.1047.CP, the Del Duca Foundation, and the H2020-MSCA-ITN-2017 program, Project ConFlex, Grant ETN-765579. 
 Furthermore, J. J. Kolumb\'an would also like to thank the Fondation Sciences Math\'ematiques de Paris for their support in the form of the PGSM Phd Fellowship.

\end{document}